\documentclass[12pt]{tran-l}
\usepackage{graphicx}
\usepackage{amsmath, amssymb, amsthm, mathrsfs,verbatim, amscd}
\usepackage[curve]{xypic}
\usepackage{setspace}
\usepackage[left=1in,top=1in,right=1in]{geometry}
\newtheorem{lemma}{Lemma}
\newtheorem{thm}{Theorem}
\newtheorem{prop}{Proposition}
\newtheorem{Def}{Definition}
\newtheorem{cor}{Corollary}

\newtheorem{proposition}{Proposition}

\newtheorem{condition}{Condition}
\newcommand{\LL}{\ensuremath{\left|\left|}}
\newcommand{\RR}{\ensuremath{\right|\right|}}
\newcommand{\R}{\ensuremath{\mathbb{R}} }
\newcommand{\C}{\ensuremath{\mathbb{C}} }
\newcommand{\N}{\ensuremath{\mathbb{N}} }

\newcommand{\Z}{\ensuremath{\mathbb{Z}} }
\newcommand{\D}{\ensuremath{\mathscr{D}} }

\newcommand{\T}{\mathcal{T}}

\renewcommand{\S}{\mathcal{S}}

\newcommand{\cf}{\mathbf{1}}
\renewcommand{\P}{\mathbb{P}}
\newcommand{\E}{\mathbb{E}}

\renewcommand{\deg}{\textrm{deg}}

\newcommand{\shape}{\mathrm{shape}}
\newcommand{\labels}{\#\mathrm{labels}}
\newcommand{\ordered}{\#\mathrm{ordered}}

\newcommand{\deq}{\mathrel{\mathop :}=}
\numberwithin{equation}{section}
\newcommand{\var}{\textrm{var}}
\newcommand{\Ht}{\textrm{ht}}
\newcommand{\M}{\mathcal{M}_w}

\renewcommand{\cR}{\mathcal{R}}
\newcommand{\fl}[1]{\lfloor #1 \rfloor}
\newcommand{\cP}{\mathcal{P}}

\begin{document}

\title{Schr\"oder's problems and scaling limits of random trees} 

\author{Jim Pitman}
\address{Department of Statistics \\ University of California\\ Berkeley, CA 94720}
\email{pitman@stat.berkeley.edu}
\thanks{JP supported in part by NSF Grant No. 0806118}

\author{Douglas Rizzolo}
\address{Department of Mathematics \\ University of Washington\\ Seattle, WA 98105}
\email{drizzolo@math.washington.edu}
\thanks{DR supported in part by NSF Grant No. 0806118, in part by the National Science Foundation Graduate  Research Fellowship under Grant No. DGE 1106400, and in part by NSF DMS-1204840.}

\subjclass[2010]{Primary 60C05, 60J80}

\begin{abstract}
In his now classic paper \cite{Schr70}, Schr\"oder posed four combinatorial problems about the number of certain types of bracketings of words and sets.  Here we address what these bracketings look like on average.  For each of the four problems we prove that a uniform pick from the appropriate set of bracketings, when considered as a tree, has the Brownian continuum random tree as its scaling limit as the size of the word or set goes to infinity.
\end{abstract}

\maketitle

\section{Introduction} 
In his now classic paper \cite{Schr70}, Schr\"oder posed four combinatorial problems about bracketings of words and sets: how many binary bracketings are there of a word of length $n$? how many bracketings are there of a word of length $n$? how many binary bracketings are there of a set of size $n$? and how many bracketings are there of a set of size $n$?  These questions are well studied and \cite{Stan99} gives a good account of the solutions.  In this paper we are concerned with a probabilistic variation on these questions: for each of the above questions, if you select a bracketing uniformly at random what does it look like?  To answer these questions, we will use the well known correspondence between the bracketings described above and various types of trees.  We will then apply Aldous's theory of continuum trees, originally developed in the series \cite{Aldo91a, Aldo91b, Aldo93} and subsequently studied by many authors, to study the scaling limits of these trees.  Let us briefly describe the correspondence between bracketings and trees.

\textbf{The first problem:} The correspondence is best illustrated by example. For $n=4$ the binary word bracketings are 
\[ (xx)(xx) \quad x(x(xx)) \quad ((xx)x)x \quad x((xx)x) \quad (x(xx))x.\]
A binary bracketing of a word with $n$ letters corresponds to rooted ordered binary tree with $n$ leaves in a natural way.  This is most easily described if we put brackets around the entire word and each letter, which are left out of our example because they are visually cumbersome.  The tree corresponding to a bracketing is constructed recursively.  A single bracketed letter is a leaf.  For a word with more than one letter, the bracketing of the whole word is the root.  Attached as subtrees to the root are, in order of appearance, the trees corresponding to the maximal proper bracketed subwords.  For $n=4$, this is illustrated by Figure \ref{fig lab bin tree}.

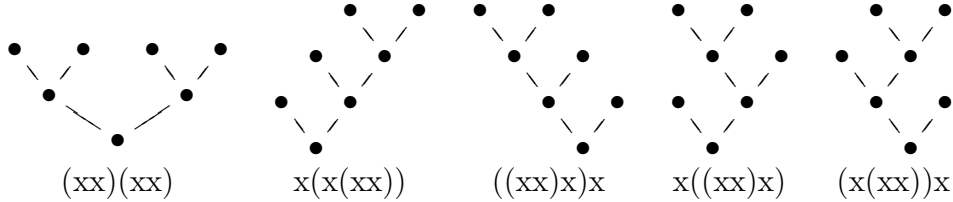
\begin{figure}[here] 
\begin{center}\begin{tabular}{ccccc} $
\xymatrix@R=5pt@C=0pt{ \\ \bullet \ar@{-}[dr]& & \bullet\ar@{-}[dl] & & \bullet\ar@{-}[dr] & &\bullet\ar@{-}[dl] \\
& \bullet \ar@{-}[drr] & & & & \bullet \ar@{-}[dll] \\
& & & \bullet 
}$ 
&
$
\xymatrix@R=5pt@C=0pt{  
& & \bullet \ar@{-}[dr] & & \bullet\ar@{-}[dl]\\
& \bullet \ar@{-}[dr] & & \bullet\ar@{-}[dl]\\
\bullet\ar@{-}[dr] & & \bullet \ar@{-}[dl]\\
& \bullet
}$
&
$
\xymatrix@R=5pt@C=0pt{  
\bullet \ar@{-}[dr] & & \bullet\ar@{-}[dl]\\
  & \bullet \ar@{-}[dr] & & \bullet\ar@{-}[dl]\\
 & & \bullet\ar@{-}[dr] & & \bullet \ar@{-}[dl]\\
& & & \bullet
}$
&
$
\xymatrix@R=5pt@C=0pt{  
 \bullet \ar@{-}[dr] & & \bullet\ar@{-}[dl]\\
& \bullet \ar@{-}[dr] & & \bullet\ar@{-}[dl]\\
\bullet\ar@{-}[dr] & & \bullet \ar@{-}[dl]\\
& \bullet
}$
&
$
\xymatrix@R=5pt@C=0pt{  
& \bullet \ar@{-}[dr] & & \bullet\ar@{-}[dl]\\
 \bullet \ar@{-}[dr] & & \bullet\ar@{-}[dl]\\
& \bullet\ar@{-}[dr] & & \bullet \ar@{-}[dl]\\
& & \bullet
}$ \\
(xx)(xx) &  x(x(xx)) &  ((xx)x)x & x((xx)x) & (x(xx))x
\end{tabular}
\end{center}
\caption{Binary word bracketings and rooted ordered binary trees for $n=4$}
\label{fig lab bin tree}
\end{figure}
It is worth noting that these trees are in bijection with rooted ordered trees with $n$ vertices, but this correspondence is not as natural as the one above.

\textbf{The second problem:} General word bracketings are defined similarly to binary word bracketings and correspond to rooted ordered trees with $n$ leaves and no vertices with out degree equal to one.  We remark that these trees were recently studied in \cite{CuKo12} due to their connection with non-crossing plane configurations.

\textbf{The third problem:} The trees associated to binary set bracketings are constructed similarly to those associated to binary word bracketings.  They are rooted, unordered, leaf-labeled binary trees.  Figure \ref{fig bin tree} shows a sample of the correspondence for $n=4$ (for $n=4$ there are $15$ bracketings, so showing the whole correspondence is unwieldy).
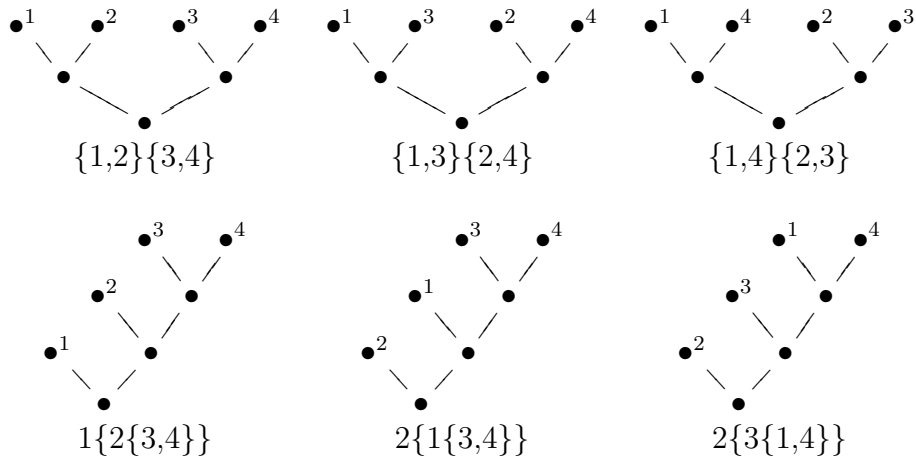
\begin{figure}[here]
\begin{center}\begin{tabular}{ccc} $
\xymatrix@R=5pt@C=0pt{ \\ \bullet^1 \ar@{-}[dr]& & \bullet^2\ar@{-}[dl] & & \bullet^3\ar@{-}[dr] & &\bullet^4\ar@{-}[dl] \\
& \bullet \ar@{-}[drr] & & & & \bullet \ar@{-}[dll] \\
& & & \bullet 
}$ 
&
$
\xymatrix@R=5pt@C=0pt{ \\ \bullet^1 \ar@{-}[dr]& & \bullet^3\ar@{-}[dl] & & \bullet^2\ar@{-}[dr] & &\bullet^4\ar@{-}[dl] \\
& \bullet \ar@{-}[drr] & & & & \bullet \ar@{-}[dll] \\
& & & \bullet 
}$ 
&
$
\xymatrix@R=5pt@C=0pt{ \\ \bullet^1 \ar@{-}[dr]& & \bullet^4\ar@{-}[dl] & & \bullet^2\ar@{-}[dr] & &\bullet^3\ar@{-}[dl] \\
& \bullet \ar@{-}[drr] & & & & \bullet \ar@{-}[dll] \\
& & & \bullet 
}$ 
 \\
\{1,2\}\{3,4\} &  \{1,3\}\{2,4\}  &  \{1,4\}\{2,3\}   \\
$
\xymatrix@R=5pt@C=0pt{ \\  
& & \bullet^3 \ar@{-}[dr] & & \bullet^4\ar@{-}[dl]\\
& \bullet^2 \ar@{-}[dr] & & \bullet\ar@{-}[dl]\\
\bullet^1\ar@{-}[dr] & & \bullet \ar@{-}[dl]\\
& \bullet
}$
&
$
\xymatrix@R=5pt@C=0pt{ \\  
& & \bullet^3 \ar@{-}[dr] & & \bullet^4\ar@{-}[dl]\\
& \bullet^1 \ar@{-}[dr] & & \bullet\ar@{-}[dl]\\
\bullet^2\ar@{-}[dr] & & \bullet \ar@{-}[dl]\\
& \bullet
}$
&
$
\xymatrix@R=5pt@C=0pt{ \\  
& & \bullet^1 \ar@{-}[dr] & & \bullet^4\ar@{-}[dl]\\
& \bullet^3 \ar@{-}[dr] & & \bullet\ar@{-}[dl]\\
\bullet^2\ar@{-}[dr] & & \bullet \ar@{-}[dl]\\
& \bullet
}$
\\
1\{2\{3,4\}\} & 2\{1\{3,4\}\} & 2\{3\{1,4\}\}
\end{tabular}
\end{center}
\caption{Binary set bracketings and rooted unordered leaf-labeled binary trees for $n=4$}
 \label{fig bin tree}
\end{figure}

\textbf{The fourth problem:} General set bracketings are defined similarly to binary set bracketings and correspond to rooted unordered leaf-labeled trees with $n$ leaves and no vertices with out degree equal to one.  In the literature, these trees are also called fragmentation trees \cite{HMPW08} and hierarchies \cite{FlSe09}.  The correspondence for $n=3$ is in Figure \ref{fig tree}.

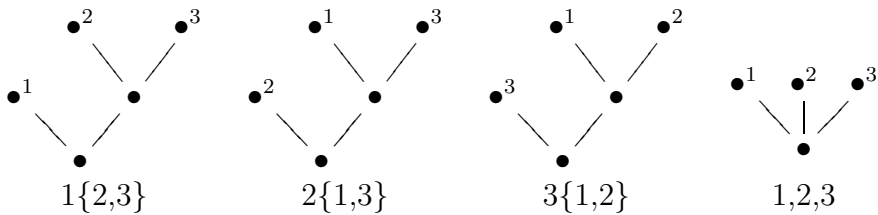
\begin{figure}[here]
\begin{center}\begin{tabular}{cccc} 
$
\xymatrix@R=10pt@C=5pt{ 
& \bullet^2 \ar@{-}[dr] & & \bullet^3\ar@{-}[dl]\\
\bullet^1\ar@{-}[dr] & & \bullet \ar@{-}[dl]\\
& \bullet
}$
&
$
\xymatrix@R=10pt@C=5pt{
& \bullet^1 \ar@{-}[dr] & & \bullet^3\ar@{-}[dl]\\
\bullet^2\ar@{-}[dr] & & \bullet \ar@{-}[dl]\\
& \bullet
}$
&
$
\xymatrix@R=10pt@C=5pt{   
& \bullet^1 \ar@{-}[dr] & & \bullet^2\ar@{-}[dl]\\
\bullet^3\ar@{-}[dr] & & \bullet \ar@{-}[dl]\\
& \bullet
}$
&
$
\xymatrix@R=10pt@C=5pt{ \\ 
 \bullet^1 \ar@{-}[dr] &\bullet^2 \ar@{-}[d] & \bullet^3\ar@{-}[dl]\\
& \bullet
}$
\\
1\{2,3\} & 2\{1,3\} & 3\{1,2\} & 1,2,3
\end{tabular}
\end{center}
\caption{Set bracketings and rooted unordered leaf-labeled trees for $n=3$}
 \label{fig tree}
\end{figure} 

Scaling limits of uniform picks from the trees appearing in the first and third problems are well studied.  A uniform pick from rooted ordered binary tree's with $n$ leaves has the same distribution as a Galton-Watson tree with offspring distribution $\xi_0=\xi_2=1/2$ conditioned to have $2n-1$ vertices.  Thus it falls within the scope of the results in \cite{Aldo93}.  Similarly, a uniform pick from rooted unordered leaf-labeled binary trees with $n$ leaves is a uniform binary fragmentation tree with $n$ leaves, and scaling limits of these are studied in \cite{HMPW08}.   In this paper we present a unified approach that is able to handle all four of these types of trees simultaneously.  Our method is essentially to link the trees appearing in these four problems to Galton-Watson trees conditioned on their number of leaves.  In particular, we obtain the following result, which is proved in Section \ref{subsec schroder}.

\begin{thm}\label{thm gw schroder}
Define four probability measures $\xi^{(1)},\dots,\xi^{(4)}$ on $\{0,1,\dots\}$ by $\xi^{(1)}_0=\xi^{(1)}_2=1/2$,
\[\xi^{(2)}_0 =2-\sqrt{2}\approx 0.5858,\quad \xi^{(2)}_1=0,\quad \mbox{and}\quad\xi^{(2)}_i = \left(\frac{2-\sqrt{2}}{2}\right)^{i-1}\approx (0.2929)^{i-1} \,\, \mbox{for }\, i\geq 2,\]
$\xi^{(3)}=\xi^{(1)}$, and 
\[\xi^{(4)}_0 = \frac{2\log(2)-1}{\log(2)},\quad \xi^{(4)}_1=0,\quad \textrm{and } \quad \xi^{(4)}_i = \frac{(\log (2))^{i-1}}{i!} \textrm{ for } i\geq 2.\] 
For each $i$ and $n$, let $T^i_n$ be distributed like a Galton-Watson tree with offspring distribution $\xi^{(i)}$ conditioned to have $n$ leaves.  
\begin{enumerate}
\item $T^1_n$ and $T^2_n$ are distributed like uniform random picks from the trees in Schr\"oder's first and second problem respectively. 
\item Let $U_n$ be a uniform random ordering of $\{1,\dots,  n\}$, independent of the trees and for $i\in \{3,4\}$ let $\tilde T^i_n$be constructed from $T^i_n$ by labeling the leaves of $T^i_n$ from left to right by $U$ and then forgetting the order structure.   Then $\tilde T^3_n$ and $\tilde T^4_n$ are distributed like uniform random picks from the trees in Schr\"oder's third and fourth problem respectively.
\end{enumerate}
\end{thm}

Scaling limits for these Galton-Watson trees were recently proven in \cite{Rizz11}, with an alternate approach given independently in \cite{Kort12}, and as a result we obtain the following theorem, the notation for which will be fully explained later.  

\begin{thm}\label{thm schroder limits}
For $i=1,2,3,4$, let $T^i_n$ be a uniform random tree of the type appearing in Schr\"oder's $i$'th problem with $n$ leaves.  For each $i$ and $n$ equip $T^i_n$ with the graph metric where edges have length one and the uniform probability measure on its leaves.  We then have the following limits with respect to the rooted Gromov-Hausdorff-Prokhorov topology:

\[\begin{array}{ccccc}
(i) & \displaystyle \frac{1}{\sqrt{n}} T^1_n \overset{d}{\to} 2\sqrt{2}T^{Br} & & (ii) &\displaystyle \frac{1}{\sqrt{n}} T^2_n \overset{d}{\to} \frac{\sqrt{2}}{2\sqrt{\sqrt{2}-1}}T^{Br} \\
& & \\
(iii) & \displaystyle \frac{1}{\sqrt{n}} T^3_n \overset{d}{\to} 2\sqrt{2}T^{Br} & & (iv) &\displaystyle \frac{1}{\sqrt{n}} T^4_n \overset{d}{\to} \frac{2}{\sqrt{4\log(2)-2}}T^{Br},
\end{array}\]
where $T^{Br}$ is the Brownian continuum random tree.
\end{thm}

As noted above, parts (i) and (iii) were originally proven in \cite{Aldo93} and \cite{HMPW08} respectively.  Parts (ii) and (iv) appear to be new, though an alternate approach to (ii) was independently obtained in \cite{Kort12}.  Though, as mentioned, this theorem follows from Theorem \ref{thm gw schroder} and \cite[Theorem 1]{Rizz11}, we include an independent proof in Section \ref{sec scaling limits}.  The proof we give here exploits the fact the distributions $\xi^i$ in Theorem \ref{thm gw schroder} have some exponential moments and is considerably simpler than the proof of \cite[Theorem 1]{Rizz11}.  Furthermore, our approach lets us obtain asymptotic results for other quantities associated to these trees as indicated in Section \ref{sec scaling limits}. 

The paper is organized as follows.  In Section \ref{sec basic models} we rigorously introduce the models of random trees under consideration here.  In Section \ref{sec scaling limits} we introduce the analytic setting for Theorem \ref{thm schroder limits} and end with the proof of this theorem.  This section also includes a detailed analysis of the depth-first processes associated with these trees.  Finally, in Section \ref{sec direct computations} we use elementary methods from analytic combinatorics to compute some asymptotic properties of these trees explicitly.

\section{Combinatorial models and Galton-Watson trees}\label{sec basic models}
In this section we develop several combinatorial and probabilistic models of trees.  There are two primary types of trees we will be dealing with in the sequel: rooted ordered unlabeled trees and rooted unordered leaf-labeled trees.  Combinatorial relations between rooted ordered unlabeled trees and rooted unordered labeled trees are well known when the size of a tree is its number of vertices (se e.g. \cite{Pitm97, Aldo91b, FlSe09, Drmo09}).  In this section we develop analogous relations when the size of a tree is its number of leaves.  Particularly important for us is Corollary \ref{cor leaf-labeled gw}, which relates Schr\"oder's problems to particular Galton-Watson trees conditioned on their number of leaves.

We briefly give an account of the formal constructions of the trees we will be considering.  Fix a countably infinite set $S$; we will consider the vertex sets of all graphs discussed to be subsets of $S$.  Let $\T^{(\ell)}_n$ denote the set of rooted unordered trees with $n$ leaves (where the root is considered a leaf if and only if it is the only vertex in the the tree) whose leaves are labeled by $\{1,2,\dots,n\}$.  More precisely, we consider the set $\T^S_n$ of all trees whose vertex sets are contained in $S$ that have a distinguished root and $n$ leaves (where the root is considered a leaf if and only if it is the only vertex in the the tree), whose leaves are labeled by $\{1,2,\dots,n\}$ and set $\T^{(\ell)}_n = \T^S_n/\sim$ where $t\sim s$ if there is a root and label preserving isomorphism from $t$ to $s$.  This is the only time we shall go through this formal construction, but all other sets of trees we discuss should be considered as formally constructed in an analogous fashion.  We also let $\T^{(\ell)} = \cup_{n\geq 1} \T^{(\ell)}_n$.  We let $\T^{(o)}_n$ be the set of rooted ordered unlabeled trees with $n$ leaves and $\T^{(o)} = \cup_{n\geq 1} \T^{(o)}_n$.

We will be proving analogous results for trees in $\T^{(\ell)}$ and $\T^{(o)}$ where the only differences in the statements and proofs will be whether the superscript is $(\ell)$ or $(o)$.  To avoid repetition we will use $\T^*$ and $\T_n^*$ to mean that the statements and proofs are valid both when all of the $*$'s are replaced by $(\ell)$'s and when they are replaced by $(o)$'s.  For a tree $t\in \T^*$, we define $|t|$ to be the number of leaves in $t$ and $\# t$ to be the number of vertices in $t$.

\subsection{Probabilities on trees}\label{subsec probs}  In this subsection we introduce models that are analogous to the simply generated trees introduced by Meir and Moon \cite{MeMo78}, but were the size of a tree is its number of leaves rather than its number of vertices.  Let $\zeta =(\zeta_i)_{i\geq 0}$ be a sequence of numbers.  We may then define the weight of a tree $t\in \T^{*}$ to be 
\[w_\zeta(t) = \prod_{v\in t} \zeta_{\deg(v)}.\]
Here and throughout, $\deg(v)$ is the out degree of $v$, i.e., the number of children of $v$.  We will assume the following conditions:
\begin{condition}  \label{cond weights}
(i) $\zeta_i\geq 0$ for all $i$, (ii) $\zeta_0>0$, and (iii) for each $n$ we have $\sum_{t\in \T^*_n} w_\zeta (t) <\infty$.
\end{condition}

Observe that if (i) and (ii) are satisfied, then (iii) is also satisfied whenever $\zeta_1=0$, as is the case for Schr\"oder's problems.  For each $n$ such that $w_\zeta(t)>0$ for some $t\in \T^{*}_n$ we may define a probability measure on $\T^{*}_n$ by
\[Q_n^{\zeta*}(t) = \frac{w_\zeta(t)}{\sum_{s\in \T^{*}_n} w_\zeta (s)}.\]
We wish to consider generating functions, but we want an ordinary generating function for $\T^{(o)}$ and an exponential generating function for $\T^{(\ell)}$.  In order to do this all at once, for $z\in \C$, we define $y^{(\ell)}_n(z) = z^n/n!$ and $y^{(o)}_n(z) = z^n$, both for $n\geq 0$, and we use $y^*_n$ in the same fashion as $\T^*$.  The weighted generating function induced on $\T^{*}$ by $\zeta$ with the weights defined above is
\[ C^*_\zeta(z) = \sum_{t\in \T^{*}} w_\zeta(t) y^*_{|t|}(z).\]
Letting $G_{\zeta,*}(z) = \sum_{i=1}^\infty \zeta_iy^*_i(z)$, it is then easy to see that $C^*_\zeta$ satisfies the functional equation
\begin{equation}\label{eq recursion} C^*_\zeta(z) = \zeta_0 z + G_{\zeta,*}(C^*_\zeta(z)),\end{equation}
in the sense of formal power series.  Our interest is in the measures $Q^{\zeta *}_n$ and, in particular, we would like to find a Galton-Watson tree $T$ such that $Q^{\zeta(o)}_n$ is the law of $T$ conditioned to have exctly $n$ leaves.  Recall that if $(\xi_i)_{i\geq 0}$ is a distribution on $\Z_+$ with mean less than or equal to one and $\xi_0>0$, a Galton-Watson tree with offspring distribution $\xi$ is a random element $T$ of $\T^{(o)}$ with law
\[\P(T=t) = \prod_{v\in t} \xi_{\deg(v)}.\]
$T$ is called critical if $\xi$ has mean equal to one.  

\begin{proposition}\label{proposition gw to Q}
If $\zeta$ is a probability distribution with mean less or equal to one and $T$ is a Galton-Watson tree with offspring distribution $\zeta$, then the law of $T$ conditioned to have exactly $n$ leaves is $Q^{\zeta (o)}_n$.
\end{proposition}

This leads to the notion of tilting, which is similar to exponential tilting for Galton-Watson trees conditioned on their number of vertices (see \cite[p. 11]{Drmo09}).

\begin{prop}\label{prop tilting}
Suppose that $\zeta$ satisfies Condition \ref{cond weights} and suppose that $a,b>0$.  Define $\tilde \zeta$ by 
\[\tilde\zeta_0=a\zeta_0 \quad \textrm{and} \quad  \ \tilde\zeta_i =  b^{i-1}\zeta_i  \ \textrm{for } \ i\geq 1.\]
Then $Q^{\zeta *}_n=Q^{\tilde\zeta*}_n$ for all $n\geq 1$.
\end{prop}

\begin{proof}
This follows immediately from the computation that, for $t\in \T^{*}_n$, $w_{\tilde \zeta}(t) = a^nb^{n-1}w_\zeta(t)$.
\end{proof}
A consequence of this is that we can find a Galton-Watson tree $T$ such that $Q^{\zeta(o)}_n$ is the law of $T$ conditioned to have $n$ leaves if we can find $a,b>0$ such that 
\[a\zeta_0+ \frac{G_{\zeta,(o)}(b)}{b}=1.\]
Furthermore, $T$ will be critical if ${G_{\zeta,(o)}}'(b)=1$.  An immediate consequence of this is the following corollary.

\begin{cor}\label{cor uni ord tree} Let $\xi^{(2)}$ be as defined in Theorem \ref{thm gw schroder}.  Note that $\xi^{(2)}$ has mean $1$ and variance $4\sqrt{2}$.  Let $T$ be a Galton-Watson tree with offspring distribution $\xi^{(2)}$.  Then the law of $T$ conditioned to have $n$ leaves is uniform on the subset of $\T^{(o)}_n$ of trees with no vertices of out degree one.
\end{cor}

\begin{proof}
The proof follows immediately from the discussion above by noting that, if $\zeta_i=1$ for $i\neq 1$ and $\zeta_1=0$ then then $Q^{\zeta(o)}_n$ is uniform on $\T^{(o)}_n$.  Explicitly, the distribution $\xi^{(2)}$ is found by solving $G_{\zeta,(o)}'(b)=1$, setting $a=(b-G_{\zeta,(o)}(b))/b$,
and tilting as in Proposition \ref{prop tilting}.
\end{proof}

Given the similarities in the construction of $Q^\zeta_n$ and $Q^{\zeta(o)}_n$, there should be a natural way to go back and forth between them.

\begin{prop}\label{prop leaf-labeling} Suppose that $\zeta$ satisfies Condition \ref{cond weights} for $*=(o)$.  Define $\hat \zeta$ by $\hat \zeta_n  = n! \zeta_n$.  Then $\hat \zeta$ satisfies Condition  \ref{cond weights} for $*=(\ell)$.  Suppose that $T$ is distributed like $Q^{\zeta(o)}_n$ and let $U$ be a uniformly random ordering of $\{1,\dots, n\}$ independent of $T$.  Define $\hat T \in \T^{(\ell)}_n$ to be the tree obtained from $T$ by labeling the leaves of $T$ by $U$ and forgetting the ordering of $T$.  Then $\hat T$ is distributed like $Q^{\hat\zeta(\ell)}_n$.
\end{prop}
    
Results of this type connecting plane and labeled trees where the size of a tree is given by the number of its vertices can be traced back to \cite{Kolc77, Pavl78, Pavl83}.  See \cite{Pitm97} for a more complete history.  Our proposition is analogous to an implicit discussion in \cite{Aldo91a, Aldo91b} as well as Theorem 7.1 in \cite{Pitm97}, which considered the case where the size of a tree is given by the number of its vertices.  To prove this proposition, we will need some notation.  For a rooted ordered tree $x$ let $\shape(x)$ be the rooted unordered tree obtained by forgetting the order on $x$.  Similarly, for $t\in \T^{(\ell)}$, $\shape(t)$ is defined to be the rooted unlabeled tree obtained from forgetting the labeling of $t$.  For $t\in \T^{(\ell)}$, $x\in \T^{(o)}$, and a rooted unordered tree $y$ define $\labels_t(x)$ to be the number of ways to label the leaves of $x$ such that when order on $x$ is forgotten the resulting tree is $t$ and $\ordered(y)$ to be the number of ordered trees whose shape is $y$.  Observe that $\labels_t(x)$ depends only on $\shape(x)$, so we will abuse our notation and write $\labels_t(\shape(x))$.

\begin{proof}
Let $t$ be an element of $\T_n^{(\ell)}$.  Observe that
\[\P(\hat T= t) = \sum_{x\in \T_n^{(o)}} \P(T=x)\P(\hat T=t | T=x).\] 
Furthermore, observe that
\[\P(\hat T=t | T=x) = \frac{\labels_t(\shape(x))}{n!},\]
Observe that $\labels_t(\shape(x))=0$ unless $\shape(t)=\shape(x)$.  Furthermore, $\P(T=x)$ depends only on $\shape(x)$, and is given by
\[\P(T=x) = \frac{\prod_{v\in \shape(x)} \zeta_{\deg(v)}}{\sum_{s\in \T^{(o)}_n} w_\zeta(s)}.\]
Consequently we have
\begin{equation}\label{eq gw 1} \P(\hat T= t)=\frac{\ordered(\shape(t)) \left(\prod_{v\in \shape(t)} \zeta_{\deg(v)}\right)\frac{\labels_t(\shape(t))}{n!}}{\sum_{s\in \T^{(o)}_n} w_\zeta(s)}.\end{equation}
But 
\begin{equation}\label{eq gw 2}\ordered(\shape(t))\labels_t(\shape(t))= \prod_{v\in \shape(t)}(\deg(v)!).\end{equation}
This is because both sides count the number of distinct leaf-labeled ordered trees that equal $t$ upon forgetting their order.  On the left hand side, count by picking an ordered tree and then labeling it and, on the right hand side, count by labeling an unordered tree with the appropriate shape and then ordering the children of each vertex.

Therefore we have
\[ \P(\hat T= t) = \frac{ w_{\hat{\zeta}}(t)}{n!\sum_{s\in \T^{(o)}_n} w_\zeta(s)}.\]
The last step is to observe that 
\[n!\sum_{s\in \T^{(o)}_n} w_\zeta(s) = \sum_{s\in \T^{(\ell)}_n} w_{\hat\zeta}(s).\]
This is because for $s\in \T^{(o)}_n$, there are $n!$ rooted ordered leaf-labeled trees whose ordered tree is $s$ upon forgetting the labeling, so the left hand side is the weighted number of rooted ordered leaf-labeled trees with $n$ leaves.  Furthermore, we have already noted above that for $s\in \T^{(\ell)}_n$, there are $\prod_{v\in s}(\deg(v)!)$ rooted ordered leaf-labeled trees whose labeled tree is $s$ upon forgetting the ordering.  Thus the right hand side is also the weighted number of rooted ordered leaf-labeled trees with $n$ leaves.  Note that this step also shows that $\hat \zeta$ satisfies Condition \ref{cond weights} for $*=(\ell)$.
\end{proof}

Combining with tilting, we have the following corollary.

\begin{cor}\label{cor leaf-labeled gw} Let $\zeta$ satisfy Condition \ref{cond weights} with $*=(\ell)$ and $\zeta_0=1$.  Suppose there exist $r>0$ and $s>0$ satisfying $s=r+G_{\xi,(\ell)}(s)$ and $G_{\xi,(\ell)}'(s)\leq 1$.  Define $\xi=(\xi_i)_{i=0}^\infty$ by $\xi_0 = rs^{-1}$ and $\xi_j = s^{j-1}\zeta_j/j!$ for $j\geq 1$.  Note that $\xi$ is a probability distribution on $\Z_+$.  Let $T$ be a Galton-Watson tree with offspring distribution $\xi$ and construct $\hat T$ by labeling the leaves of $T$ uniformly at random with $\{1,\dots, |T|\}$, independently of $T$ and then forget the order of $T$.  Then $\P(\hat T\in \cdot | |T|=n) = Q^{\zeta(\ell)}_n(\cdot)$ for all $n\geq 1$ such that $Q^{\zeta(\ell)}_n$ is defined.  Furthermore, for $n$ such that $Q^{\zeta(\ell)}_n$ is not defined, $P(|T|=n)=0$.
\end{cor}

\subsection{Schr\"oder's problems}\label{subsec schroder}
In this section we record which of the trees above correspond to the trees that appear in Schr\"oder's problems.  The proofs of the claims here are simple applications of the results in Section \ref{subsec probs}.

\textbf{The first problem:}  The trees here are uniform binary rooted ordered unlabeled trees.  We can obtain these by taking $*=(o)$ and $\zeta_0=\zeta_2=1$ and $\zeta_i=0$ for $i\notin \{0,2\}$.  Letting $\xi$ be the probability distribution given by $\xi_0=\xi_2=1/2$ and $T$ be a Galton-Watson tree with offspring distribution $\xi$, we have that $T$ conditioned to have $n$ leaves is a uniform binary rooted ordered unlabeled tree with $n$ leaves.  Also note that $T$ is critical and the variance of $\xi$ is equal to one.

\textbf{The second problem:}  These are uniform rooted ordered trees with no vertices of out degree one.  These were dealt with in Corollary \ref{cor uni ord tree}

\textbf{The third problem:}  These are uniform binary unordered leaf-labeled trees.  We can obtain these by taking $*=(\ell)$ and $\zeta_0=\zeta_2=1$ and $\zeta_i=0$ for $i\notin \{0,2\}$.  In this case, if $T$ is the Galton-Watson tree defined in the first problem and $\hat T$  is defined as in Corollary \ref{cor leaf-labeled gw}, then $\hat T$ conditioned to have $n$ leaves is a uniform binary unordered leaf-labeled tree with $n$ leaves.

\textbf{The fourth problem:}  These are uniform rooted unordered leaf-labeled trees with no vertices with out-degree $1$.  We can obtain these by taking $*=(\ell)$ and $\zeta_1=0$ and $\zeta_i=1$ for $i\neq 1$.  We define a probability distribution $\xi^{(4)}$ as in Theorem \ref{thm gw schroder}.  Note that $\xi^{(4)}$ has mean $1$ and variance $\var(\xi^{(4)}) = 2\log 2$.  Letting $T$ be a Galton-Watson tree with offspring distribution $\xi^{(4)}$ and defining $\hat T$  is as in Corollary \ref{cor leaf-labeled gw}, we have that $\hat T$ conditioned to have $n$ leaves is a uniform unordered leaf-labeled tree with no vertices of out degree one and $n$ leaves.

\subsection{Gibbs trees}
Above we saw a natural way to put probability measures on $\T^{(\ell)}_n$ that are concentrated on fragmentation trees (the trees appearing in Schr\"oder's fourth problem); namely, take $\zeta_1=0$.  Another natural type of probability to put on fragmentation trees is a \textit{Gibbs model}, which we now describe.  First, we need to set up the natural framework in which to view fragmentation trees.  The idea is that, while in Schr\"oder's fourth problem we have an arbitrary set bracketing, for fragmentations we recursively partition a set.  This dynamic view of constructing a set bracketing makes Gibbs models quite natural.

\begin{Def}[\cite{McPW08}]
A fragmentation of the finite set $B$ is a collection $\mathbf{t}_B$ of non-empty subsets of $B$ such that
\begin{enumerate}
\item $B\in \mathbf{t}_B$
\item If $\# B\geq 2$ then there is a partition of $B$ into $k\geq 2$ parts $B_1,\dots,B_k$, called the children of $B$, such that
\[\mathbf{t}_B = \{B\}\cup\mathbf{t}_{B_1}\cup\cdots\cup \mathbf{t}_{B_k},\]
where $\mathbf{t}_{B_i}$ is a fragmentation of $B_i$.
\end{enumerate}
\end{Def}

We can naturally consider $\mathbf{t}_B$ as a tree whose vertices are the elements of $\mathbf{t}_B$ and whose edges are defined by the parent-child relationship.   Considering the properties of such a tree leads naturally to the following definition of a \emph{fragmentation tree} on $B$.

\begin{Def}
A \textit{fragmentation tree} $T$ on $n$ leaves is a rooted tree such that 
\begin{enumerate}
\item The root of $T$ does not have degree $1$,
\item $T$ has no non-root vertices of degree $2$,
\item The leaves of $T$ are labeled by a set $B$ with $\#B=n$.   We denote the label of a leaf $v$ by $\ell(v)$.
\end{enumerate}
\end{Def}

The idea of the Gibbs model is that, at each step in the fragmentation the next step is distributed according to multiplicative weights depending on the block sizes.  We first take a sequence $\{\alpha_n\}$, $\alpha_n\geq 0$ of weights and a Gibbs weight, which is a function $g:\Z_+\to \R_+$ with $g(0)=0$ and $g(1)>0$.   Then, for $n\geq 2$, define a normalization constant
\[Z(n) = \sum_{\{B_1,\dots,B_k\}} \alpha_k \prod_{j=1}^kg(\#B_j),\]
where the sum is over unordered partitions of $[n]$ into at least two blocks.  Whenever we write a formula like this, we assume that each block $B_i$ is nonempty.   For $n$ such that $Z(n)>0$, define the probability of a partition of $[n]$ by 
\[ P^{g,\alpha}_n(\{B_1,\dots,B_k\}) = p(\#B_1,\dots,\#B_k) = \frac{\alpha_k\prod_{j=1}^kg(\#B_j)}{Z(n)}.\]
The probability of a fragmentation $X$ of $[n]$ is then defined as
\[ P^{g,\alpha}_n(X) = \prod_{B\in X} P^{g,\alpha}_n(\{B_1,\dots,B_k\}),\]
where $\{B_1,\dots, B_k\}$ are the children of $B$.   Using the correspondence between fragmentations and fragmentation trees, for $T_n\in \T^{(\ell)}_n$, we define $P^{g,\alpha}_n(T_n)$ to be $P^{g,\alpha}_n(X)$ where $X$ is the fragmentation determined by $T_n$.   The probabilistic properties of Gibbs models are studied in \cite{McPW08}.      

\begin{thm}
Suppose that $\zeta$ satisfies Condition \ref{cond weights} with $*=(\ell)$ and $\zeta_1=0$.  Define $\alpha_k=\zeta_k$ and $g(k) = k![z^k]C^{(\ell)}_\zeta(z)$.  Then $Z(n) = g(n)$, $Q^{\zeta(\ell)}_n$ and $P^{g,\alpha}_n$ are defined for the same values of $n$, and $Q^{\zeta(\ell)}_n = P^{g,\alpha}_n$ when they are defined.  Furthermore, given a nonnegative weight sequence $\alpha$ and a Gibbs weight $g$ such that $Z(n)=g(n)$, there is a $\zeta$ satisfying Condition \ref{cond weights} with $*=(\ell)$ and $\zeta_1=0$ such that $Q^{\zeta(\ell)}_n = P^{g,\alpha}_n$.
\end{thm}

\begin{proof} Since the number of partitions of $[n]$ into $k$ ordered nonempty blocks with sizes $n_1,\dots, n_k$ is given by the multinomial coefficient, we see that for $n\geq 2$
\[Z(n) = \sum_{\{B_1,\dots,B_k\}} \alpha_k \prod_{j=1}^kg(\#B_j) =\sum_{k=2}^\infty \frac{\alpha_k}{k!}\sum_{(n_1,\dots,n_k)\in \N^k\atop n_1+\cdots+n_k=n}{ n\choose n_1,\dots,n_k}\prod_{j=1}^kg(n_j),\]
where the $1/(k!)$ appears because the partitions $\{B_1,\dots, B_k\}$ are unordered.  Note also that one way to count the weighted number of trees of size $n$ is to decompose by the degree of the root and the sizes of the subtrees attached to the root.  Doing so yields the formula  
\[n![z^n]C^{(\ell)}_\zeta(z)  =\sum_{k=2}^\infty \frac{\alpha_k}{k!}\sum_{(n_1,\dots,n_k)\in \N^k\atop n_1+\cdots+n_k=n}{ n\choose n_1,\dots,n_k}\prod_{j=1}^k\left(n_j![z^{n_j}]C^{(\ell)}_\zeta(z)\right). \]
Since $g(k) = k![z^k]C^{(\ell)}_\zeta(z)$ by definition, it follows that $Z(n) = g(n)$.   Using this, one proves inductively that $P^{g,\alpha}_n(T_n) = Q^{\zeta(\ell)}_n(T_n)$.   Furthermore, observe that the condition $Z(n)=g(n)$ implies that there is a weight sequence $(\zeta_i)_{i\geq 0}$ from which the fragmentation model can be derived in the above manner; just take $\zeta_0=g(1)$, $\zeta_1=0$, and $\zeta_k = \alpha_k$ for $k\geq 2$.
\end{proof}

When we have $Z(n) = g(n)$, the model is called a \textit{combinatorial} Gibbs model.  This is justified by the fact that, in this case, $Z(n)$ (and thus $g(n)$) is the weighted number of trees with $n$ leaves.  For example, if we let $g(n)$ be the number of fragmentation trees with $n$ leaves, and $\alpha_k=1$ for $k\geq 2$, we then see that 

\[Z(n) = \sum_{\{B_1,\dots,B_k\}} \prod_{j=1}^kg(\#B_j).\]
The right hand side of this equation is just the sum over partitions at the root of a fragmentation tree with $n$ leaves of the number of fragmentation trees with that partition at the root, which is precisely the number of fragmentation trees with $n$ leaves.  That is, $Z(n)=g(n)$.

Note that combinatorial Gibbs models are a generalization of the hierarchies studied in  \cite{FlSe09} and, as previously observed, a special case of the Gibbs models introduced in \cite{McPW08}.

\section{Scaling limits}\label{sec scaling limits}
We now turn to scaling limits of the models of trees we have been discussing.  Fortunately for us, the heavy lifting has already been done in \cite{Rizz11}.  In order to use the results from that paper, we must first introduce the formalism required to handle limits of random metric measure spaces.

\subsection{Trees as metric measure spaces}
The trees we have been talking about can naturally be considered as metric spaces with the graph metric.  That is, the distance between to vertices is the number of edges on the path connecting them.  Let $(t,d)$ be a tree equipped with the graph metric.  For $a>0$, we define $at$ to be the metric space $(t,ad)$, i.e. the metric is scaled by $a$.  This is equivalent to saying the edges have length $a$ rather than length $1$ in the definition of the graph metric.  More, generally we can attach a positive length to each edge in $t$ and use these to in the definition of the graph metric.  Moreover, the trees we are dealing with are rooted so we consider $(t,d)$ as a pointed metric space with the root as the point.  Moreover, we are concerned with the leaves, so we attach a measure $\mu_{ t}$, which is the uniform probability measure on the leaves of $t$.  If we have a random tree $T$, this gives rise to a random pointed metric measure space $(T,d,\textrm{root},\mu_{T})$.  To make this last concept rigorous, we need to put a topology on pointed metric measure spaces.  This is hard to do in general, but note that the pointed metric measure spaces that come from the trees we are discussing are compact.

Let $\M$ be the set of equivalence classes of compact pointed metric measure spaces (equivalence here being up to point and measure preserving isometry).  It is worth pointing out that $\M$ actually is a set in the sense of ZFC, though this takes some work to show.  We metrize $\M$ with the pointed Gromov-Hausdorff-Prokhorov metric (see \cite{HaMi12}).  Fix $(X,d,\rho,\mu), (X',d',\rho,\mu') \in \M$ and define
\[ d_{\textrm{GHP}}(X,X') = \inf_{(M,\delta)} \inf_{\phi:X\to M \atop \phi' :X'\to M}\left[ \delta(\phi(\rho),\phi'(\rho')) \vee d_H(\phi(X),\phi'(X'))\vee d_P(\phi_*\mu,\phi'_*\mu')\right],\]
where the first infimum is over metric spaces $(M,\delta)$, the second infimum if over isometric embeddings $\phi$ and $\phi'$ of $X$ and $X'$ into $M$, $d_H$ is the Hausdorff distance on compact subsets of $M$, and $d_P$ is the Prokhorov distance between the pushforward $\phi_*\mu$ of $\mu$ by $\phi$ and the pushforward $\phi'_*\mu'$ of $\mu'$ by $\phi'$.  Again, the definition of this metric has potential to run into set-theoretic difficulties, but they are not terribly difficult to resolve.

\begin{prop}[Proposition 1 in \cite{HaMi12}] 
The space $(\M, d_{\textrm{GHP}})$ is a complete separable metric space.
\end{prop}

An $\R$-tree is a complete metric space $(T,d)$ with the following properties:
\begin{itemize}
\item For $v,w\in T$, there exists a unique isometry $\phi_{v,w}:[0,d(v,w)]\to T$ with $\phi_{v,w}(0)=v$ and $\phi_{v,w}(d(v,w))=w$.
\item For every continuous injective function $c:[0,1]\to T$ such that $c(0)=v$ and $c(1)=w$, we have $c([0,1]) = \phi_{v,w}([0,d(v,w)])$.
\end{itemize}

If $(T,d)$ is an $\R$-tree, every choice of root $\rho\in T$ and probability measure $\mu$ on $T$ yields an element $(T,d,\rho,\mu)$ of $\M$.  With this choice of root also comes a height function $\Ht(v) = d(v,\rho)$.  The leaves of $T$ can then be defined as a point $v\in T$ such that $v$ is not in $[[\rho,w[[ \deq \phi_{\rho,w}([0,\Ht(w)))$ for any $w\in T$.  The set of leaves is denoted $\mathcal{L}(T)$.  

\begin{Def}
A continuum tree is an $\R$-tree $(T,d,\rho,\mu)$ with a choice of root and probability measure such that $\mu$ is non-atomic, $\mu(\mathcal{L}(T))=1$, and for every non-leaf vertex $w$, $\mu\{v\in T : [[\rho,v]]\cap [[\rho,w]] = [[\rho,w]]\}>0$.
\end{Def}

The last condition says that there is a positive mass of leaves above every non-leaf vertex.  We will usually just refer to a continuum tree $T$, leaving the metric, root, and measure as implicit.  A continuum random tree (CRT) is an $(\M,d_{GHP})$ valued random variable that is almost surely a continuum tree.

\subsection{The Brownian continuum random tree}
Continuous functions give a nice way of constructing $\R$-trees.  Suppose that $f:[0,1]\to \R_+$ is continuous and $f(0)=f(1)=0$.  We can define a pseudo-metric on $[0,1]$ by $d_f(a,b) = f(a)+f(b) - 2\min_{a\leq t\leq b} f(t)$ for $a\leq b$.  Define an equivalence relation by $a\sim b$ if and only if $d_f(a,b)=0$.  Letting $T_f=[0,1]/\sim$, we obtain a metric space $(T_f,d_f)$.  Theorem 2.1 in \cite{DuLe05} tells us that $(T_f,d_f)$ is a compact $\R$-tree.  Letting $\rho_f:[0,1] \to T_f$ be the natural map taking a point to its equivalence class, we can take $\rho_f(0)$ as the root of $T_f$.  For a probability measure $\mu$ on $[0,1]$, we then obtain an element of $\M$ by equipping $T_f$ with the pushforward $\mu_f$ of $\mu$ by $\rho_f$.  

\begin{Def} Let $(e(t),0\leq t\leq 1)$ be standard Brownian excursion.  The Brownian continuum random tree, denoted $T^{Br}$, is the continuum random tree $(T_e,d_e,\rho_e(0),\lambda_e)$, where $\lambda$ is Lebesgue measure on $[0,1]$. 
\end{Def}

Note that the elementary properties of Brownian excursion imply that $T^{Br}$ actually is almost surely a continuum tree.  It is also worth noting that in our formalism the Brownian continuum random tree originally defined by Aldous in \cite{Aldo93} corresponds to $(T_{2e},d_{2e},\rho_e(0),\mu_{2e})$, but the convention has since shifted to the one we have adopted here (see e.g. \cite{Gall06}).  The following lemma shows that this construction is above board, measure theoretically speaking.  We let $C_+[0,1]$ denote the set of continuous functions from $[0,1]$ to $[0,\infty)$ that map $0$ and $1$ to $0$. 

\begin{lemma}\label{lemma functional continuity} 
The map $F(f,\mu) = (T_f, d_f,\rho_f(0),\mu_f)$ is a continuous map from $C_+[0,1]\times \cP([0,1])$ equipped with the product of supremum topology with the topology of weak convergence to $\M$ with the Gromov-Hausdorff-Prokhorov topology.
\end{lemma}   

This lemma has been in the folklore for quite some time, with variants dating back to \cite[Theorem 23]{Aldo93}.  However, as far as we can tell, no formal statement or proof appears in the literature so we include one here.

\begin{proof}
Fix $(f,\mu), (g,\nu)\in C_+[0,1]\times \cP([0,1])$.  Define a relation $\cR$ on $T_f\times T_g$ by 
\[\cR = \{ (x,y) \in T_f\times T_g \ : \ \exists s\in [0,1] \textrm{ such that } \rho_f(s)=x \textrm{ and } \rho_g(s)=y\}.\] 
The distortion of $\cR$ is defined as
\[\textrm{dis}(\cR) = \sup\{ |d_f(x,x')-d_g(y,y')| : (x,y)\in \cR \ , \ (x',y')\in \cR\}.\]
We define a metric on the disjoint union $Z:=T_f\sqcup T_g$ of $T_f$ and $T_g$ by $d_Z(x,y)= d_h(x,y)$ if $x,y\in T_h$ for $h=f,g$, 
\[ d_Z(x,y) = \inf\left\{ d_f(x,x') +\frac{1}{2}\textrm{dis}(\cR) + d_g(y,y')  : (x',y')\in \cR\right\},\]
if $x\in T_f$ and $y\in T_g$, and extend by symmetry.  For the remained of the proof we identify $T_f$ and $T_g$ with their natural embeddings in $\Z$.  Observe that $d_H(T_f,T_g) \leq \textrm{dis}(\cR)/2 \leq 2||f-g||$ and $d_Z(\rho_f(0),\rho_g(0))=\textrm{dis}(\cR)/2$ since $(\rho_f(0),\rho_g(0))\in \cR$.  

To finish proving the continuity of $F$, it remains to show that the Prokhorov distance between $\mu_f$ and $\nu_g$ is can be made small if $(f,\mu)$ is sufficiently close to $(g,\nu)$.  For $h>0$, we define
\[ \omega(f,h) = \sup\{|f(x)-f(y)| : |x-y|<h\}\]
to be the $h$-modulus of continuity of $f$.  For any subset $B$ of a metric space $(E,\delta)$ and $\epsilon>0$, we define $B^\epsilon = \{x\in E : d(x,B)<\epsilon\}$.  The two key observations are that for $r>\textrm{dis}(\cR)/2$ and $I\subseteq [0,1]$ we have $\rho_g(I) \subseteq \rho_f(I)^{r}$ and if $\kappa,\epsilon_0>0$ we have $\rho_g(I^\kappa) \subseteq \rho_g(I)^{2\omega(g,\kappa)+\epsilon_0}$.  Combining these, we see that if $A\subset Z$,
\[\begin{split} \rho_f^{-1}(A)^\kappa \subseteq \rho_g^{-1}(\rho_g(\rho_f^{-1}(A)^\kappa)) & \subseteq \rho_g^{-1}(\rho_g(\rho_f^{-1}(A))^{2\omega(g,\kappa)+\epsilon_0})\\
&  \subseteq \rho_g^{-1}(\rho_f(\rho_f^{-1}(A))^{2\omega(g,\kappa)+r}) \\
&\subseteq \rho_g^{-1}(A^{2\omega(g,\kappa)+r}). \end{split}\]
Consequently, if $d_P(\mu,\nu)<\kappa$ and $A$ is measurable we have
\[ \mu_f(A) = \mu(\rho_f^{-1}(A)) \leq \nu(\rho_f^{-1}(A)^\kappa)+\kappa \leq \nu(\rho_g^{-1}(A^{2\omega(g,\kappa)+r})) + \kappa = \nu_g(A^{2\omega(g,\kappa)+r})+\kappa.\]
Similarly, $\nu_g(A) \leq \mu_f(A^{2\omega(f,\kappa)+r})+\kappa$.  Since $\omega(\cdot, h)$ is continuous on $C_+[0,1]$, it is easy to see from these inequalities that $d_P(\mu_f,\nu_g)$ can be made small by making $||f-g|| + d_P(\mu,\nu)$ small.
\end{proof}

\subsection{Excursions of random walks}
The basis of our approach to scaling limits of Galton-Watson trees with $n$ leaves is a new conditioned limit theorem for excursions of random walks.  Let $\mu$ be a probability distribution on $\{-1,0,1,2,\dots\}$ that has mean $0$ and finite, nonzero, variance $\sigma^2$.  Further let $\{X_i\}_{i\geq 1}$ be i.i.d. distributed like $\mu$.  We will restrict ourselves to the canonical situation where the $X_i$ are the coordinate functions on $\R^{\N}$ and $\N=\{1,2,\dots\}$.  For $\mathbf{x}\in \R^\N$ define $\tau_{-1}(\mathbf{x}) = \inf\{ i : x_i=-1\}$.  Let $S_0=0$ and $S_n= \sum_{i=1}^n X_i$ and let $N = \inf\{i : X_i=-1\} = \inf\{i: S_i-S_{i-1}=-1\}$.  Define $N^0=0$ and for $k\geq 1$ let $N^k = N^{k-1} + N \circ \theta_{N^{k-1}}$ where $\theta$ is the shift operator.  Observe that $(N^i-N^{i-1})_{i\geq 1}$ is an i.i.d. sequence with $\gamma:= \E N = 1/\mu(\{-1\})$.  Also note $N$ has geometric distribution with parameter $\mu(\{-1\})$, so that $\E[\exp(\lambda N)]<\infty$ for some $\lambda>0$.

For $\mathbf{x}\in \R^\N$, we let $x_0=0$ and define $x_n(t)$, for $0\leq t\leq 1$, as the $n$'th time scaled linearly interpolated process
\[ x_n(t) = x_{[nt]} + (nt-[nt])(x_{[nt]+1}-x_{[nt]}),\]
where $[t]$ is the integer part of $t$.

\begin{thm}\label{theorem passage bridge limit}
Assume that $\P(N^n = \tau_{-1}(\{S_i\}_{i\geq 0})) >0$ for all sufficiently large $n$ and recall that $\gamma := \E N$.  For $n\geq 1$, define the law $\P_n$ on $C[0,1]$ by
\[ \P_n(A) = \P\left(\frac{1}{\sigma\sqrt{\gamma n}}S_{N^n}(t) \in A \ \Big|  \tau_{-1}(\{ S_i\}_{i\geq 0})=N^n\right).\]
Let $\mathbb{W}^{\mathrm{ex}}$ be the law of standard (positive) Brownian excursion.  We then have $\P_n \Rightarrow \mathbb{W}^{\mathrm{ex}}$ as $n\to \infty$.  That is, $\P_n$ converges weakly to $\mathbb{W}^{\mathrm{ex}}$ in $C[0,1]$.
\end{thm}  

\begin{proof} Define $\tilde X_i = \sum_{N^{i-1}+1}^{N^i} X_i$.  Note that the $\tilde X_i$ are i.i.d. and, by Wald's equation, have mean $0$ and finite, nonzero variance $\sigma^2\E N$.  Let $\tilde S_0=0$ and $\tilde S_n = \sum_{i=1}^n \tilde X_i$.  Observe that 
\[\{N^n = \tau_{-1}(\{S_i\}_{i\geq 0})\} = \{n=\tau_{-1}(\{\tilde S_i\}_{i\geq 0})\}.\] 
A consequence of this is that 
\[\P(N^n = \tau_{-1}(\{S_i\}_{i\geq 0})) = \P(n=\tau_{-1}(\{\tilde S_i\}_{i\geq 0}))=O(n^{-3/2}),\]
with the last equality being a standard consequence of the Otter-Dwass formula and the local limit theorem (see e.g. \cite{Rizz11}).

We consider these processes as elements of $C[0,1]$ with the uniform topology.  Since $\E\exp(\lambda N)<\infty$ for some $\lambda >0$, Petrov's lemmas (as formulated in \cite[Appendix A.1]{MaMo03}) show that for each $\epsilon>0$ there exist constants $c_1,c_2>0$ such that 
\[\P\left(\LL\frac{1}{\gamma n}N^n(t)-t\RR >\epsilon\right) \leq c_1\exp\left(-c_2 n\right).\] 
It follows that
\[\P\left(\LL\frac{1}{N^n}N^n(t)-t\RR >\epsilon\right) \leq 2c_1\exp\left(-c_2 n\right).\]
Let $\Psi_n(t)= (N^n)^{-1}N^n(t)$.  Consequently we have that
\begin{equation}\label{equation time change}\P\left(\LL\Psi_n^{-1}(t)-t\RR >\epsilon\right) \leq  \P\left(\LL\Psi_n(t)-t\RR >\epsilon\right) \leq 2c_1\exp\left(-c_2 n\right).\end{equation}
Hence 
\[ \P\left(\LL\Psi_n^{-1}(t)-t\RR >\epsilon \ |\ \tau_{-1}(\{\tilde S_i\}_{i\geq 0})=n \right) = O(n^{3/2}\exp\left(-c_2 n\right)).\]
Using the standard (if seldom written, see e.g. \cite{Kers98}) fact that 
\[ \P\left(\frac{1}{\sigma \sqrt{\gamma n}} \tilde S_n(t) \in \cdot \ \Big| \ \tau_{-1}(\{\tilde S_i\}_{i\geq 0})=n\right) \Rightarrow \mathbb{W}^{\mathrm{ex}},\]
we thus have the joint convergence 
\[\P\left[ \left(\frac{1}{\sigma \sqrt{\gamma n}} \tilde S_n(t), \Psi_n^{-1}\right)  \in \cdot \ \Big| \ \tau_{-1}(\{\tilde S_i\}_{i\geq 0})=n\right] \Rightarrow \mathbb{W}^{\mathrm{ex}} \times \delta_{(t, \ 0\leq t\leq 1)}.\]
Since the composition map $(f,g)\mapsto f\circ g$ is a continuous map from $C([0,1],\R)\times C([0,1],[0,1])$ to $C([0,1],\R)$, it follows from the continuous mapping theorem that
\[ \P\left(\frac{1}{\sigma \sqrt{\gamma n}} \tilde S_n(\Psi_n^{-1}(t)) \in \cdot \ \Big| \ \tau_{-1}(\{\tilde S_i\}_{i\geq 0})=n\right) \Rightarrow \mathbb{W}^{\mathrm{ex}}.\]
Now, observe that for $0\leq  k \leq n$ we have 
\[ \tilde S_n(\Psi_n^{-1}(N^k/N^n)) =  \tilde S_n(k/n) = \tilde S_k = S_{N^k} = S_n(N^k/N^n).\]
Therefore, to deduce the convergence of the first passage bridges of $(\sigma\sqrt{\gamma n})^{-1}S_{N^n}$ from the convergence of the first passage bridges of $(\sigma\sqrt{\gamma n})^{-1}\tilde S_n\circ\Psi^{-1}$, we need only control how the processes differ in time intervals of the form $[N^k/N^n,N^{k+1}/N^n]$.  We will do this in terms of the modulus of continuity of $\tilde S_n$.
Since $\mu$ is supported on $\{-1,0,1,\dots\}$ and $N=\inf\{i : S_i-S_{i-1}=-1\}$, for all $1\leq k \leq n$ we have
\[ \max_{N^{k-1}<i\leq N^{k}} \left|\sum_{j=N^{k-1}+1}^{i} X_j\right| \leq \sum_{j=N^{k-1}+1}^{N^{k}} X_j + 1 = \tilde X_k +1.\]
Suppose $\delta>0$ and suppose that $n>\delta^{-1}$.  We then have that 
\[  \max_{1\leq k \leq n} |\tilde X_k| \leq \omega(\tilde S_n,\delta) ,\]
where we recall that
\[\omega(f,\delta) = \sup\{ |f(u)-f(v)| : |u-v|\leq \delta\}\]
is the $\delta$-modulus of continuity.  Consequently
\[ \LL S_n(t) - \tilde S_n(\Psi_n^{-1}(t))\RR \leq \omega(\tilde S_n,\delta) + 1.\]  
Let $\epsilon>0$ be given.  By the standard condition for tightness in $C[0,1]$ (see e.g. \cite[Theorem 8.2]{Bill99}), there exists $\delta>0$ such that
\[\limsup_{n\to\infty} \P\left(\omega\left(\frac{1}{\sigma \sqrt{\gamma n}}\tilde S_n, \delta\right) \geq \epsilon/2 \ \Big| \ \tau_{-1}(\{\tilde S_i\}_{i\geq 0})=n\right) <\epsilon.\]
It follows that
\[ \lim_{n\to\infty} \P\left(\LL \frac{1}{\sigma \sqrt{\gamma n}} \tilde S_n(\Psi_n^{-1}(t)) - \frac{1}{\sigma \sqrt{\gamma n}} S_n(t) \RR \geq \epsilon \ \Big| \ \tau_{-1}(\{ S_i\}_{i\geq 0})=N^n\right)  =0.\]
We conclude, using e.g. \cite[Theorem 4.1]{Bill99}, that $\P_n\Rightarrow \mathbb{W}^{\mathrm{ex}}$.
\end{proof}

\subsection{Limits of Galton-Watson trees} 
For $x\in \Z^\N= \Z^{\{1,2,3,\dots\}}$, let $\tau_{-1} = \inf\left\{ n : \sum_{i=1}^n x_n=-1\right\}$.  Let $\D$ be the set of sequences of first passage bridge increments in $\Z^\N$ that are bounded below by $-1$.  Formally,
\[\D = \left\{ x \in \Z^\N \ :   x_{i} \geq -1 \textrm{ for } i\geq 1,  \ x_i=0 \textrm{ for } i\geq \tau_{-1}(x), \ \tau_{-1}(x)<\infty\right\}.\]  
Suppose that $t\in \T^{(o)}$ has $n$ vertices.  The depth-first walk of $t$ is a function $f^t:\{0,\dots, 2n\}\to t$ defined by $f^t(0)$ is the root of $t$ and $f^t(i)$ is the smallest child of $f^t(i-1)$ that is not in $\{f^t(0),\dots, f^t(i-1)\}$ if one exists and the parent of $f^t(i-1)$ otherwise.  Index the vertices $V$ of $t$ from $1$ to $n$ by order of appearance on the depth-first walk of $t$, so that $V=\{v_1,\dots, v_n\}$.  Define $DQ^t=(DQ_k^t)_{k=1}^\infty$ by $DQ_k^t = \deg \ v_k-1$ for $k\leq n$ and $0$ for $k>n$, which are the increments of the depth-first queue of $t$.    Note that $DQ^t\in \D$.  Furthermore $t\mapsto DQ^t$ is a bijection from $\T^{(o)}$ to $\D$ (see e.g. \cite{Pitm06}).  $DQ^t$ is the list of increments of the depth-first queue of $t$, which is defined by 
\[Q_k^t = \sum_{i=1}^k DQ_i^t.\]
We will also be interested in several other processes associated to $t$.  Two of them are easily described in terms of the depth-first order of the vertices.  They are the contour and height processes of $t$ which are defined by
\[ C_k^t = d(\textrm{root}(t), f^t(k)) \quad \textrm{and} \quad H_k^t =  d(\textrm{root}(t), v_k)\]
respectively.  Two others are breadth-first processes.  The breadth-first order of the vertices of a tree $t$ with $n$ vertices is defined as follows: Let $n_k$ be the number of vertices of $t$ at distance $k$ from the root and let $ht(t)$ be the height of $t$.  We define $v_1'$ to be the root of $t$ and for $1\leq k \leq ht(t)$ we define $(v'_{n_{k-1}+1},\dots, v'_{n_{k}})$ to be the vertices of $t$ at height $k$ listed from left to right.  The complete list $(v'_1,\dots, v'_n)$ is the vertices of $t$ listed in breadth-first order.  The breadth-first queue is defined by 
\[ B^t_k = 1+\sum_{i=1}^k (\deg(v'_i)-1).\]
The level profile of $t$ is defined by $L^t_k=n_k$.  Observe that the level profile and breadth-first queue are related by 
\[L^t_k = B^t(1+n_1 +\cdots + n_{k-1}).\]
See \cite[p. 7]{Kers98} for the details regarding this relation.

Let $(\xi_i)_{i\geq 0}$ be a probability distribution on $\Z_+$ with mean $1$ and $0<\xi_0<1$.  Suppose that $T$ is a Galton-Watson tree with offspring distribution $\xi$.  Let $\mathbf{X}=(X_i)_{i\geq 1}$ be i.i.d. with distribution $\P(X_1=i)=\xi(\{i+1\})$.  From the fact that $t\mapsto DQ^t$ is a bijection, it follows straight forwardly that $(DQ^T_k,k=1,\dots, \#T) =_d (X_i, 1\leq i\leq \tau_{-1}(\{S_i\}_{i\geq 0}))$ and $\P(|T|= n) = \P(N^n= \tau_{-1}(\{S_i\}_{i\geq 0})) = O(n^{-3/2})$.  For $n$ such that $\P(|T|=n)>0$, let $T_n$ be distributed like $T$ conditioned to have $n$ leaves.  These observations lead us to the following proposition.

\begin{prop}\label{proposition random walk depth queue}
For $n$ such that $\P(| T|=n)>0$, we have
\[ (Q^{T_n}_k,0\leq k\leq \#T_n) =_d (S_0(\mathbf{X}),\dots, S_{\tau_{-1}}(\mathbf{X})) \big| (N^n= \tau_{-1}(\{S_i\}_{i\geq 1})) \]
\end{prop} 

Similar arguments show that

\begin{prop}
For $n$ such that $\P(|T|=n)>0$, we have
\[ (B^{T_n}_k,0\leq k\leq \#T_n) =_d (S_0(\mathbf{X})+1,\dots, S_{\tau_{-1}}(\mathbf{X})+1) \big| (N^n= \tau_{-1}(\{S_i\}_{i\geq 1})) \]
\end{prop}

Now, if we were to follow our previous notational convention, the time scaled linear interpolation of $Q^{T_n}$ would be denoted by $Q^{T_n}_{\#T_n}$, which seems congested.  We simplify this by dropping the superscript and the hash sign in the subscript.  That is, we use the notation, for $0\leq s\leq 1$,
\[Q_{T_n}(s) :=Q^{T_n}_{\#T_n}(s) =  Q^{T_n}_{[\#T_ns]} + (\#T_ns-[\#T_ns])(Q^{T_n}_{[\#T_ns]+1}-Q^{T_n}_{[\#T_ns]}).\]
We define $B_{T_n}$ and $H_{T_n}$ similarly, while we define $C_{T_n}$ with a slight modification by
\[C_{T_n}(s) :=C^{T_n}_{2\#T_n}(s) =  C^{T_n}_{[2\#T_ns]} + (2\#T_ns-[2\#T_ns])(C^{T_n}_{[2\#T_ns]+1}-C^{T_n}_{[2\#T_ns]}).\]
Restating Theorem \ref{theorem passage bridge limit} we obtain the following.

\begin{thm}\label{theorem depth-first queue}
Let $(\xi_i)_{i\geq 0}$ be a probability distribution on $\Z_+$ with mean $1$ and $0< \var(\xi):=\sigma^2<\infty$.  Suppose that $T$ is a Galton-Watson tree with offspring distribution $\xi$.  Assume that for all sufficiently large $n$ we have $\P(\# T= n)>0$.  For such $n$, let $T_n$ be distributed like $T$ conditioned to have $n$ leaves.  We then have the convergences in distribution
\[ \frac{\sqrt{\xi_0}}{\sigma \sqrt{n}} Q_{T_n} \Rightarrow e \quad \textrm{and}\quad  \frac{\sqrt{\xi_0}}{\sigma \sqrt{n}} B_{T_n} \Rightarrow e \]
in $C[0,1]$, where $e$ has distribution $\mathbb{W}^{\mathrm{ex}}$.
\end{thm}

Note that this says nothing about the convergence of the joint distribution of the scaled depth and breadth first queues, which we leave as an open problem.

Assuming $\xi$ has some exponential moments, we can strengthen this result.

\begin{thm}\label{theorem depth-first processes}
In addition to the conditions of Theorem \ref{theorem depth-first queue}, assume that $\int \exp(\alpha x) \xi(dx)<\infty$ for some $\alpha>0$.  Let $(e(t),0\leq t\leq 1)$ have distribution $\mathbb{W}^{ex}$.  We then have the convergence in distribution in $C([0,1]^3)$
\[ \left( \frac{\sqrt{\xi_0}}{\sigma \sqrt{n}} Q_{T_n}(s_1), \ \frac{\sqrt{\xi_0}}{\sigma \sqrt{n}} H_{T_n}(s_2), \frac{\sqrt{\xi_0}}{\sigma \sqrt{n}} C_{T_n}(s_3)\right)_{[0,1]^3} \Rightarrow \left(e(s_1), \ \frac{2}{\sigma^2} e(s_2), \ \frac{2}{\sigma^2}e(s_3)\right)_{[0,1]^3}.\]
\end{thm}

A similar result with weaker hypotheses appears in \cite[Theorem 5.9]{Kort12}, and the result for the scaled contour function can also be derived from \cite[Theorem 1]{Rizz11}.  However, for us, it follows immediately from Theorem \ref{theorem depth-first queue} and the next theorem.  By exploiting the existence of exponential moments in our setting, we are able to provide a much simpler proof than those appearing in \cite{Kort12,Rizz11}.  Because of this, Theorem \ref{theorem depth-first processes} can also be used to simplify the approach to non-crossing configurations of the plane in \cite{CuKo12}, which makes use of this theorem applied to the trees in Schr\"oder's second problem.

\begin{thm}\label{theorem exponential closeness}
For each $\nu>0$ there exist constants $N$ and $\alpha>0$ such that for $n\geq N$ we have
\begin{enumerate}
\item[(i)] $\displaystyle \P\left(\LL Q_{T_n} - \frac{\sigma^2}{2}H_{T_n} \RR \geq (\#T_n)^{1/4+\nu}\right)  \leq e^{-\alpha n^\nu}$ and
\item[(ii)] $\displaystyle \P\left(\LL Q_{T_n} - \frac{\sigma^2}{2}C_{T_n} \RR \geq (\#T_n)^{1/4+\nu}\right)  \leq e^{-\alpha n^\nu}$.
\end{enumerate}
\end{thm}

\begin{proof}
Theorems 2 and 3 in \cite{MaMo03} prove the analogous result when $T_n$ is distributed like $T$ conditioned to have $n$ vertices.  Our theorem follows from those by decomposition by the number of vertices in $T_n$.  For example, (ii) follows from Theorem 2 in \cite{MaMo03} by
\begin{eqnarray*} \lefteqn{\displaystyle \P\left(\LL Q_{T_n} - \frac{\sigma^2}{2}C_{T_n} \RR \geq (\#T_n)^{1/4+\nu}\right) }\\
 &&\displaystyle \qquad \leq  \frac{1}{\P(|T|=n)}\sum_{k\geq n \atop \P( \#T=k)>0} \P\left(\LL Q_{T} - \frac{\sigma^2}{2}C_{T} \RR \geq k^{1/4+\nu} \ \Big| \ \#T=k\right)\\
&&\displaystyle\qquad\leq  \ cn^{3/2}\sum_{k=n}^\infty e^{-\alpha' k^\nu} \\
&&\displaystyle\qquad \leq  \ e^{-\alpha n^\nu},
\end{eqnarray*}
where the second inequality and $\alpha'>0$ are given by Theorem 2 in \cite{MaMo03}.
\end{proof}

\begin{cor}\label{corollary level profile}
Maintaining the notation and hypotheses from Theorem \ref{theorem depth-first processes}, we have
\[ \left(\frac{\sqrt{\xi_0}}{\sigma\sqrt{n}} L^{T_n}_{\fl{\xi_0^{1/2}\#T_ns/\sigma n^{1/2}}} , \ s\geq 0\right) \rightarrow \left(\frac{1}{2}L_{s/2} , s\geq 0\right)\]
where $(L_s)_{s\geq 0}$ is the local time of standard Brownian excursion and the convergence is in distribution in the Skorokhod space $D[0,\infty)$.
\end{cor}

\begin{proof}
The proof of Theorem 1 in \cite{Kers98} goes through almost verbatim.  We need only verify a condition involving the cumulative height process
\[ \tilde H_{T_n}(u) =   \frac{\sqrt{\xi_0}}{\sigma\sqrt{n}}  \int_0^u L^{T_n}_{[\sigma^{-1}\sqrt{\xi_0n} s]}ds =\frac{\#T_n+1}{n} \int_0^1 \cf\left( \bar H_{T_n}(s) \leq [\sigma^{-1}\sqrt{\xi_0n} u]\right)ds,\]
where $\bar H_{T_n}(s) := H^{T_n}_{[(\#T_n+1) s]}$ for $0\leq s<1$.  What needs to be shown is that
\begin{equation}\label{equation k-condition} \lim_{\epsilon\downarrow 0} \limsup_{n\to\infty} \P\left(\tilde H_{T_n}(u) \leq \epsilon\right)=0\end{equation}
for every $u>0$.  To see this, note that the scaled convergence of $H_{T_n}$ to Brownian excursion, which is continuous, implies that $\bar H_{T_n}$ has the same scaling limit as $H_{T_n}$ in the Skorokhod space $D[0,1]$.  Since $\#T_n/n \to \xi_0^{-1}$, Equation \ref{equation k-condition} follows from standard continuity properties of the Skorokhod topology and the fact that
\[\mathbb{W}^{ex}\left( \int_0^1 \cf(e(s)\leq u)ds =0 \right)=0\]
for all $u>0$.   The proof of Theorem 1 in \cite[p. 18]{Kers98} now goes through exactly, using Theorem \ref{theorem depth-first queue}  in place of \cite[Theorem 11]{Kers98} and Equation \ref{equation k-condition} in place of \cite[Lemma 9]{Kers98}.
\end{proof} 

We now show how to obtain scaling limits for weighted trees from the limits for depth-first processes obtained in Theorem \ref{theorem depth-first processes}.

\begin{thm}\label{theorem weighted gw limits}
Let $\xi$ be an offspring distribution satisfying the hypothesis of Theorem \ref{theorem depth-first processes} and let $T_n$ be distributed like a Galton-Watson tree with offspring distribution $\xi$ conditioned to have $n$ leaves.  Consider $T_n$ as a rooted weighted metric space with the graph metric and the uniform probability measure $\mu^n$ on the leaves of $T_n$.  We have
\[ \frac{1}{\sqrt{n}}T_n \overset{d}{\to} \frac{2}{\sigma \sqrt{\xi_0}}T^{Br}\]
with respect to the Gromov-Hausdorff-Prokhorov topology. 
\end{thm}

\begin{proof}
Let $\nu^n$ be the empirical distribution of the location of leaves along the scaled contour process $n^{-1/2}C_{T_n}(s)$ of $T_n$.  It is clear that 
\[d_{GHP}\left(\left(\frac{1}{\sqrt{n}}T_n,\mu^n\right), \left(T_{\frac{1}{ \sqrt{n}} C_{T_n}}, \nu^n_{\frac{1}{ \sqrt{n}} C_{T_n}}\right) \right) \leq \frac{1}{\sqrt{n}}.\]
Therefore, by Theorem \ref{theorem depth-first processes} and Lemma \ref{lemma functional continuity}, all that remains to be shown is that $\nu^n \Rightarrow \lambda$ in probability, where $\lambda$ is Lebesgue measure on $[0,1]$.

Let $\hat\nu^n$ be the empirical distribution of the location of leaves along the height process of $T_n$, which we note is the same as the empirical distribution of the location of leaves along the depth-first queue of $T_n$.  We denote the vertices of $T_n$ listed in order of appearance on the depth-first walk of $T_n$ by $(v_1,\dots, v_{\#T_n})$.  For $1\leq l \leq \#T_n$, define $m(l) = \inf\{k : f^{T_n}(k)=v_l\}$, were we recall that $f^{T_n}$ is the depth-first walk of $T_n$.   From \cite[Lemma 2]{MaMo03}, we see that 
\[ m(l) = 2l-1- H^{T_n}_l,\]
where our formula is slightly  different from that in \cite{MaMo03} due to indexing considerations.  If we let $N^k$ denote the location of the $k$'th leaf along the depth-first queue of $T_n$, we see that 
\[\hat\nu^n = \frac{1}{n} \sum_{i=1}^n \delta_{\frac{N^i}{\#T_n}} \quad \textrm{and}\quad \nu^n = \frac{1}{n} \sum_{i=1}^n \delta_{\left(\frac{N^i}{\#T_n} - \frac{1+H^{T_n}_{N^i}}{2\#T_n}\right)}\]
It follows from Proposition \ref{proposition random walk depth queue} and Equation \ref{equation time change} that $\hat\nu^n \Rightarrow \lambda$ in probability.  Furthermore, it follows from Theorem \ref{theorem depth-first processes} that 
\[(\#T_n)^{-1}\sup_{1\leq l\leq \#T_n} H^{T_n}_l = O(1/\sqrt{n})\]
in probability.  As a result of this, we have that $\nu^n\Rightarrow\lambda$ in probability as well and this completes the proof.
\end{proof}

\section{Explicit computations using analytic combinatorics} \label{sec direct computations}
The convergence result above is a powerful theorem for obtaining asymptotics of various tree statistics, but it is difficult to prove and, as a result, asymptotics thus obtained can seem mysterious.  Consequently it is worth noting we can obtain a number of asymptotic results directly using analytic combinatorics.  This analytic approach is based on considering the asymptotics of generating functions.   The primary source for asymptotics in general is \cite{FlSe09}, which develops the theory with extensive examples.

Our main goal in this section is to develop the general framework of additive functionals for leaf-labeled trees whose size is counted by their number of leaves and use this to find the asymptotic distribution of the height of a uniformly randomly chosen leaf.  We also find the limit of the expected height of a random leaf.  These computations are meant to be illustrative and by no means exhaust the power of analytic combinatorics framework.  Indeed, it seems that most of the techniques used to study for simple varieties of trees (see \cite{FlSe09} for a summary of the extensive work in this area) have close analogs that will provide results about the trees we are considering here. 

\subsection{Analytic background}
In this subsection we recall from \cite{FlSe09} some fundamental results from analytic combinatorics.  The next subsection applies these to our setting.  The approach is based on the asymptotics of several universal functions.  Recall that if $f(z)$ is either a formal power series, $[z^n]f(z)$ denotes the coefficient of $z^n$.  Similarly, if $f:\C\to\C$ is analytic at $0$ then $[z^n]f(z)$ denotes the coefficient of $z^n$ in the power series expansion of $f$ at $0$. 

\begin{prop}
\label{prop basic functions}
Let $f(z)=(1-z)^{1/2}$, $g(z)=(1-z)^{-1/2}$, and $h(z)=(1-z)^{-1}$.   Then $[z^n]f(z) \sim  - 1/2\sqrt{\pi n^3}$, $[z^n]g(z)\sim 1/\sqrt{n\pi}$, and $[z^n]h(z)=1$.
\end{prop}

To use these classical results we need a special type of analyticity called $\Delta$-analyticity, which we now define.

\begin{Def}[Definition VI.I p. 389 \cite{FlSe09}]
Given two number $\phi$ and $R$ with $R>1$ and $0<\phi<\pi/2$, the open domain $\Delta(\phi,R)$ is defined as
\[\Delta(\phi,R) = \{z\ | \ |z|<R, \ z\neq 1, \ |\arg(z-1)|>\phi\}.\]
For a complex number $\zeta$ a domain $D$ is a $\Delta$-domain at $\zeta$ if there exist $\phi$ and $R$ such that $D=\zeta \Delta(\phi,R)$.   A function is $\Delta$-analytic if it is analytic on a $\Delta$-domain.
\end{Def}  

Define $\lambda(z)$ by
\[  \lambda(z) =\frac{1}{z}\log\frac{1}{1-z}\]
and let $\mathcal{S} = \left\{(1-z)^{-\alpha}\lambda(z)^{\beta}   \ | \ \alpha,\beta\in\C\right\}$.

\begin{thm}[Theorem VI.4 p. 393 \cite{FlSe09}]\label{thm transfer}
Let $f(z)$ be a function analytic at $0$ with a singularity at $\zeta$, such that $f(z)$ can be continued to a domain of the form $\zeta\Delta_0$, for a $\Delta$-domain $\Delta_0$.   Assume that there exist two function $\sigma$ and $\tau$, where $\sigma$ is a (finite) linear combination of elements of $\mathcal{S}$ and $\tau\in \S$, so that 
\[ f(z) = \sigma(z/\zeta)+O(\tau(z/\zeta))\quad \mbox{as } z\to \zeta \ \mbox{ in }\ \zeta\Delta_0.\]
Then the coefficients of $f(z)$ satisfy the asymptotic estimate
\[ f_n = \zeta^{-n}\sigma_n + O(\zeta^{-n}\tau_n^\star),\]
where $\tau_n^\star = n^{a-1}(\log n)^b$, if $\tau(z)=(1-z)^{-a} \lambda(z)^b$.
\end{thm}

Occasionally we will also need to deal with derivatives and the next theorem shows us how this is done.

\begin{thm}[Theorem VI.8 p. 419 \cite{FlSe09}]
\label{thm singular differentiation}
Let $f(z)$ be $\Delta$-analytic with singular expansion near its singularity of the simple form
\[ f(z) = \sum_{j=0}^J c_j(1-z)^{a_j} + O((1-z)^A),\]
with $A,a_1,a_2,\dots \in \R$.  Then, for each integer $r>0$, the derivative $f^{(r)}(z)$ is $\Delta$-analytic.   The expansion of the derivative at its singularity is obtained through term by term differentiation:
\[\frac{d^r}{dz^r}f(z) = (-1)^r \sum_{j=0}^J c_j \frac{\Gamma(a_j+1)}{\Gamma(a_j+1-r)}(1-z)^{a_j-r} + O((1-z)^{A-r}).\]
\end{thm}

The generating functions we will work with fall into the smooth implicit-function schema, which provides a way to derive coefficient asymptotics from functional equations.

\begin{Def}[Definition VII.4 p. 467 \cite{FlSe09}] \label{def: smooth-implicit}
Let $y(z)$ be a function analytic at $0$, $y(z)=\sum_{n\geq 0} y_nz^n$, with $y_0=0$ and $y_n\geq 0$.   The function is said to belong to the smooth implicit-function schema if there exists a bivariate function $G(z,w)$ such that 
\[y(z)=G(z,y(z)),\]
where $G(z,w)$ satisfies the following conditions.
\begin{enumerate}
\item[(i)] $G(z,w)=\sum_{m,n\geq 0} g_{m,n}z^mw^n$ is analytic in a domain $|z|<R$ and $|w|<S$, for some $R,S>0$.
\item[(ii)] The coefficients of $G$ satisfy $g_{m,n}\geq 0$, $g_{0,0}=0$, $g_{0,1}\neq 1$, and $g_{m,n}>0$ for some $m$ and for some $n\geq 2$.
\item[(iii)] There exist two numbers $r$ and $s$ such that $0<r<R$ and $0<s<S$, satisfying the system of equations
\[ G(r,s)=s,\quad G_w(r,s)=1,\quad \mbox{with} \,\, r<R, \,\, s<S,\]
which is called the characteristic system.
\end{enumerate}
\end{Def}

\begin{Def}[Definition IV.5 p. 266 \cite{FlSe09}]
Consider the formal power series $f(z) = \sum f_nz^n$.   The series $f$ is said to admit \emph{span} d if for some $r$
\[ \{n: f_n\neq 0\}_{n=0}^\infty \subseteq r+d\Z_+.\]
The largest span is the \emph{period} of $f$.   If $f$ has period $1$, then $f$ is \emph{aperiodic}.
\end{Def}

With this definition, we get the following theorem.   It is worth noting that this result appears in several places in the literature.   We give the version that appears as Theorem VII.3 on page 468 of \cite{FlSe09}.   In that source it is footnoted that many statements occurring previously in the literature contained errors, so caution is advised.

\begin{thm}[Theorem VII.3 p. 468 \cite{FlSe09}]
\label{thm implicit-function}
Let $y(z)$ belong to the smooth implicit-function schema defined by $G(z,w)$, with $(r,s)$ the positive solution of the characteristic system.   Then $y(z)$ converges at $z=r$, where it has a square root singularity,
\[ y(z) \underset{z\to r}{=} s-\gamma\sqrt{1-z/r}+O(1-z/r),\quad \gamma \equiv \sqrt{\frac{2rG_z(r,s)}{G_{ww}(r,s)}},\]
the expansion being valid in a $\Delta$-domain.   If, in addition, $y(z)$ is aperiodic, then $r$ is the unique dominant singularity of $y$ and the coefficients satisfy
\[[z^n]y(z)\underset{n\to\infty}{=} \frac{\gamma}{2\sqrt{\pi n^3}}r^{-n}\left(1+O(n^{-1})\right).\]
\end{thm}

We will also need the following theorem.

\begin{thm}[(A special case of) Theorem IX.16 p. 709 \cite{FlSe09}]
\label{thm semi-large powers}
Let $H(z)$ be $\Delta$-continuable and of the form $H(z) = \sigma - h(1-z/\rho)^{1/2}+O(1-z/\rho)$ and let $k_n = x_nn^{1/2}$ for $x_n$ in any compact subinterval of $(0,\infty)$.   Then
\[[z^n]H(z)^{k_n}\sim \sigma^{k_n}\rho^{-n}  \frac{h k_n}{2\sigma \sqrt{\pi n^3}}  \exp\left(-\frac{h^2k_n^2}{4\sigma^2 n}\right).\]
\end{thm}

\subsection{Restricting the generality}
So far we have been considering a very general situation.  In order to simplify the computations in the following section we will focus on a more restricted setting.  In particular, we will let $\zeta = (\zeta_i)_{i\geq 0}$ be a sequence of non-negative weights such that $\zeta_0=1$, $\zeta_1=0$, $\gcd\{k:\zeta_k\neq 0\}=1$, and
\[ G_\zeta(z) = \sum_{j=2}^\infty \zeta_j \frac{z^j}{j!}\]
is entire.  These conditions can be relaxed, but doing so makes the analysis more difficult.

\begin{prop}\label{prop: C-is-implicit}  With $\zeta$ as above, $C_\zeta$ (defined in Section \ref{subsec probs})  belongs to the smooth implicit-function schema with $G(z,w)=z+G_\zeta(w)$.  Furthermore, in the case where $\zeta$ corresponds to Schr\"oder's third problem $(r,s)=(1/2,1)$ and in the case of the fourth problem, we have $(r,s)=(2\log(2)-1,\log(2))$ .  Additionally 
\[ [z^n]C_\zeta(z) \sim \frac{\gamma}{2\sqrt{\pi n^3}}r^{-n},\quad \gamma = \sqrt{\frac{2r}{G_\zeta''(s)}}.\]
\end{prop}

\begin{proof}
All that really needs to be checked is that the characteristic system has a positive solution.   For $G(z,w)=z+G_\zeta(w)$, the characteristic system is $s=r+G_\zeta(s)$ and $G_\zeta'(s)=1$.  Using that $G_\zeta$ is entire, $G_\zeta'(0)=0$ and $G_\zeta'(+\infty)=+\infty$, and $G_\zeta'$ is increasing on $\R_+$, the intermediate value theorem yields $s>0$.  An easy computation yields that $G_\zeta(s)<sG_\zeta'(s)=s$, so $r>0$ as well.
\end{proof}

\subsection{The height of a random leaf}\label{sec main}
Let $H_n$ be the height of a randomly chosen leaf from a tree in $\T^{(\ell)}_n$.   Specifically, to get $H_n$, we choose a tree $T_n$ from $\T^{(\ell)}_n$ according to $Q^{\zeta (\ell)}_n$ and then choose a leaf uniformly at random from $T_n$.  Our main result in this section is the following theorem.

\begin{thm}
\label{thm Rayleigh limit}
As $n\to\infty$
\[ \frac{\lambda}{\sqrt{n}}H_n \overset{d}{\to} \mbox{Rayleigh}(1),\]
for $\lambda = \sqrt{G_\zeta''(s)r}$.   In Schr\"oder's third problem $\lambda = 1/\sqrt{2}$, and in the fourth problem $\lambda = \sqrt{4\log(2)-2}$.
\end{thm}

We remark that this theorem can be derived as a consequence of Theorem \ref{theorem weighted gw limits} and Propositions \ref{proposition gw to Q} and \ref{prop leaf-labeling}.  In fact, those results allow us to derive asymptotic results for the joint law of the heights of any fixed number of random leaves.  The analytic approach we take now has the advantage of requiring less background than those results and also gives an independent verification of the scaling constants obtained in those results.    

Our approach will be that of additive functionals, whose theory we now develop.  We parallel the development of these functions in \cite{FlSe09}, p. 457.   Their work was done for simple varieties of trees whose size was determined by the number of vertices.   Here we work with trees whose size is determined by the number of leaves.

For a rooted unordered tree $t$ whose leaves are labeled by $B\subseteq \N$, let $\tilde t \in \T^{(\ell)}$ be the tree that results from relabeling the leaves of $t$ by the unique increasing bijection from $B$ to $\{1,2,\dots,|B|\}$ (where $|B|$ is the cardinality of $B$).  Suppose we have functions $\xi,\theta,\psi : \T^{(\ell)} \to \R$ satisfying the relation
\[ \xi(t) = \theta(t)+\sum_{j=1}^{\deg(t)} \psi(\tilde t_j),\]
where $\deg(t)$ is the root degree of $t$ and the $\{t_j\}$ are the root subtrees of $t$ ordered in increasing order of the leaf with the smallest label.  Letting $\bullet$ denote the tree on one leaf, we note that $\deg(\bullet)=0$, so in particular $\xi(\bullet)=\theta(\bullet)$.   Define the exponential generating functions
\[ \Xi(z) = \sum_{t\in\T^{(\ell)}} \xi(t)w(t) \frac{z^{|t|}}{|t|!}, \quad \Theta(z) = \sum_{t\in\T^{(\ell)}} \theta(t)w(t) \frac{z^{|t|}}{|t|!}, \quad \mbox{and }\quad \Psi(z) = \sum_{t\in\T^{(\ell)}} \psi(t)w(t) \frac{z^{|t|}}{|t|!}.\]
Our results make use of the following lemma, which is a relation of formal power series.

\begin{lemma}
\label{le additive functions}
Let $C_\zeta(z)$ be the exponential generating function for $\mathcal{C}$.   Then
\begin{equation}
\label{eq additive}
\Xi(z) = \Theta(z) + G_\zeta'(C_\zeta(z))\Psi(z).
\end{equation}
In the purely recursive case where $\xi\equiv \psi$ we have
\begin{equation}
\label{eq recursive additive}
\Xi(z) = \frac{\Theta(z)}{1-G_\zeta'(C_\zeta(z))} = C'_\zeta(z)\Theta(z).
\end{equation}
\end{lemma}

\begin{proof}
We clearly have
\[\Xi(z) = \Theta(z)+\tilde\Psi(z),\quad \mbox{where} \quad \tilde\Psi(z) = \sum_{t\in\T^{(\ell)}}\left( w(t)\frac{z^{|t|}}{|t|!}\sum_{j=1}^{\deg(t)} \psi(\tilde t_j)\right).\]
Decomposing by root degree and using that $\xi(\bullet)=\theta(\bullet)$, we have
\[\begin{split} \tilde\Psi(z) 
& = \sum_{r\geq 1} \sum_{\deg(t)=r} \zeta_r \prod_{j=1}^r w(\tilde t_j) \frac{z^{|t_1|+\cdots+|t_r|}}{(|t_1|+\cdots+|t_r|)!} (\psi(\tilde t_1)+\cdots+\psi(\tilde t_r)) \\
& = \sum_{r\geq 1} \left[\frac{ \zeta_r}{r!} \sum_{(t_1,\dots, t_r)\in (\T^{(\ell)})^r} \prod_{j=1}^r w(t_j) {|t_1|+\cdots+|t_r| \choose |t_1|,\dots, |t_r|}\frac{z^{|t_1|+\cdots+|t_r|}}{(|t_1|+\cdots+|t_r|)!} (\psi(t_1)+\cdots+\psi(t_r))\right] \\
& = \sum_{r\geq 1} \left[\frac{\zeta_r}{r!} \sum_{(t_1,\dots, t_r)\in (\T^{(\ell)})^r} \prod_{j=1}^r w(t_j)\frac{z^{|t_1|+\cdots+|t_r|}}{|t_1|!\cdots |t_r|!}(\psi(t_1)+\cdots+\psi(t_r))\right] \\
& = \sum_{r\geq 1}  \frac{ \zeta_r}{(r-1)!} C_\zeta(z)^{r-1}\Psi(z) \\
& = G_\zeta'(C_\zeta(z))\Psi(z).
\end{split}\]
This yields \eqref{eq additive}.   In the recursive case, we have $\Xi(z)=\Theta(z)+G_\zeta'(C_\zeta(z))\Xi(z)$.   Solving for $\Xi(z)$ gives the first equality in \eqref{eq recursive additive}.   To get the second, we differentiate \eqref{eq recursion} to get $C'_\zeta(z)=1+G_\zeta'(C_\zeta(z))C'_\zeta(z)$.   Solving for $C'_\zeta(z)$ gives $C'_\zeta(z)=1/(1-G_\zeta'(C_\zeta(z)))$, from which the second equality in \eqref{eq recursive additive} is immediate.
\end{proof}

Two immediate applications are to counting the weighted numbers of leaves and vertices of a given height.  

\begin{thm} As $n\to \infty$ the expected number of leaves at height $k$ converges to $G_\zeta''(s)rk$ and the expected number of nodes at height $k$ converges to $sG_\zeta''(s)k+1$.
\end{thm}

We remark that this Theorem involves convergence of expectations and this is not a direct consequence of our results on convergence in distribution.

\begin{proof}
Let $\xi_k(t)$ be the number of leaves of height $k$ in $t$, so that $\xi_k(t)w(t)$ is the weighted number of leaves of height $k$.   Define $\Xi_k = \sum_t \xi_k(t)w(t) z^{|t|}/|t|!$.   For $k\geq 1$ we apply the lemma with $\xi=\xi_k$, $\theta=0$ and $\psi=\xi_{k-1}$ to obtain
\[\Xi_k(z) = G_\zeta'(C_\zeta(z)) \Xi_{k-1}(z),\]
which easily yields 
\[\Xi_k(z) = \left[G_\zeta'(C_\zeta(z))\right]^k\Xi_0(z)=z\left[G_\zeta'(C_\zeta(z))\right]^k.\]
Letting $\Lambda_k(z)$ be the generating function for the weighted number of vertices of height $k$ we similarly get
\[\Lambda_k(z) = \left[G_\zeta'(C_\zeta(z))\right]^k\Lambda_0(z)=C_\zeta(z)\left[G_\zeta'(C_\zeta(z))\right]^k.\]
Using these forms, we are able to compute asymptotics.   Expanding $G_\zeta'$ about $s$, we have that $G_\zeta'(z) = 1 + G_\zeta''(s)(z-s) + O((z-s)^2)$.   Plugging in the asymptotic expansion of $C_\zeta$ we get from Proposition \ref{prop: C-is-implicit} and Theorem \ref{thm implicit-function} and doing some algebra, we have
\begin{equation}\label{eq phi' of C}G_\zeta'(C_\zeta(z)) = 1-G_\zeta''(s)\gamma \sqrt{1-z/r} + O(1-z/r).\end{equation}
Hence, using that $(1-z)^k = 1-kz+O(z^2)$, we see that
\[ [G_\zeta'(C_\zeta(z))]^k  = 1-G_\zeta''(s)k\gamma\sqrt{1-z/r} + O(1-z/r).\]
Thus, using Theorem \ref{thm transfer}, we have 
\[ [z^n] \Xi_k(z) = [z^n]z\left[G_\zeta'(C_\zeta(z))\right]^k \sim \frac{G_\zeta''(s)k\gamma}{2r^{n-1}\sqrt{\pi n^3}},\]
and, similarly with a bit more algebra,
\[ [z^n] \Lambda_k(z) = [z^n]C_\zeta(z)\left[G_\zeta'(C_\zeta(z))\right]^k \sim \frac{\gamma (skG_\zeta''(s)+1)}{2r^n\sqrt{\pi n^3}}.\]
Using the result on p. 474 of \cite{FlSe09}, that $[z^n] C_\zeta(z) \sim \gamma / 2r^n\sqrt{\pi n^3}$, we find that 
\[ E_{\T^{(\ell)}_n} (\xi_k) = \frac{n![z^n] \Xi_k(z)}{n![z^n]C_\zeta(z)} \sim G_\zeta''(s) r k.\]
Letting $\zeta_k : \T^{(\ell)}\to \Z$ be the number of nodes of height $k$ in $t$, we have that 
\[E_{\T^{(\ell)}_n} (\zeta_k) =  \frac{n![z^n] \Lambda_k(z)}{n![z^n]C_\zeta(z)} \sim sG_\zeta''(s) k + 1.\]
\end{proof}

The proof of Theorem \ref{thm Rayleigh limit} is similar, but we make use of Theorem \ref{thm semi-large powers} for the asymptotics.

\begin{proof}[Proof of Theorem \ref{thm Rayleigh limit}]  Let $\{k_n\}$ be a sequence of integers varying such that $c k_n/n^{1/2} \to x\in (0,\infty)$ for some $c>0$.   By Theorem \ref{thm semi-large powers}  and equation \eqref{eq phi' of C} we see that
\begin{equation}\label{eq semi-large C} [z^n] [G_\zeta'(C_\zeta(z))]^{k_n} \sim r^{-n} \frac{G_\zeta''(s)\gamma }{2 \sqrt{\pi n^3}} k_n \exp\left(-\frac{G_\zeta''(s)^2\gamma^2k_n^2}{4 n}\right).\end{equation}
Therefore
\[ [z^n] \Xi_{k_n}(z) \sim r^{-(n-1)} \frac{G_\zeta''(s)\gamma }{2 \sqrt{\pi n^3}} k_{n-1} \exp\left(-\frac{G_\zeta''(s)^2\gamma^2k_{n-1}^2}{4 (n-1)}\right).\]
This yields
\[ \begin{split} E_{\T^{(\ell)}_n}(\xi_{k_n}) & \sim G_\zeta''(s)rk_{n-1}\exp\left(-\frac{G_\zeta''(s)^2\gamma^2k_{n-1}^2}{4 (n-1)}\right) \\
& =G_\zeta''(s)rk_{n-1}\exp\left(-\frac{G_\zeta''(s)rk_{n-1}^2}{2 (n-1)}\right)
.\end{split}\]
Note that $P(H_n=k) = E_{\T^{(\ell)}_n}(\xi_{k})/n$.      Observe that $\{k_n\}$ satisfies the hypotheses of the above theorem.   Consequently, we have that 
\[\begin{split} \frac{\sqrt{n}}{c} P_n\left(\frac{c}{\sqrt{n}} H_n = \frac{c}{\sqrt{n}}  k_n\right) & =  \frac{\sqrt{n}}{c} P(H_n=k_n)\\
& =  \frac{1}{c\sqrt{n}} E_{\T^{(\ell)}_n}(\xi_{k_n}) \\
& \sim \frac{1}{c^2} G_\zeta''(s)r \frac{ ck_{n-1}}{\sqrt{n}}\exp\left(-\frac{G_\zeta''(s)r}{2c^2} \frac{c^2k_{n-1}^2}{ (n-1)}\right) \\
& \to \frac{G_\zeta''(s) r}{c^2} x\exp\left(-\frac{G_\zeta''(s)r}{2c^2} x^2\right)
.\end{split}\]
The proof is finished by an application of a standard corollary of Scheff\'e's theorem (see Theorem 3.3 in \cite{Bill99} for an idea of the proof, just adapted for a distribution on $(0,\infty)$) and choosing $c=\sqrt{G_\zeta''(s)r}$.
\end{proof}

In addition to proving convergence in distribution we can prove convergence of the first moment.  We remark that this Theorem involves convergence of expectations and this is not a direct consequence of our results on convergence in distribution.

\begin{thm}\label{thm expected leaf height} $E_{\T^{(\ell)}_n}H_n \sim  \sqrt{\frac{\pi}{2rG_\zeta''(s)}}  n^{1/2}$.
\end{thm}

The approach is to first compute the expected sum of the heights of the leaves of a tree, since, if $\phi(t)$ is the sum of the heights of the leaves of $t$, then $E_{\T^{(\ell)}_n}H_n = \frac{1}{n} E_{\T^{(\ell)}_n}\phi $.  Therefore Theorem \ref{thm expected leaf height} is an immediate consequence of the next result.

\begin{thm} Let $\phi(t)$ be the sum of the heights of the leaves of $t$.  Then $E_{\T^{(\ell)}_n}\phi \sim   \sqrt{\frac{\pi}{2rG_\zeta''(s)}} n^{3/2}$.
\end{thm}

\begin{proof} Observe that 
\[ \phi(t) = |t| +\sum_{j=1}^{\deg(T)} \phi(\tilde t_j).\]
Let $\Phi(z)$ be the exponential generating function associated with $\phi$.   Applying Lemma \ref{le additive functions}, we have
\[ \Phi(z) = z(C'_\zeta(z))^2.\]
By Theorem \ref{thm singular differentiation} we have
\[ C'_\zeta(z) = \frac{\gamma}{2r}(1-z/r)^{-1/2} + O(1).\]
Consequently,
\[ (C'_\zeta(z))^2 = \frac{\gamma^2}{4r^2}(1-z/r)^{-1} + ((1-z/r)^{-1/2}).\]
Therefore 
\[ E_{\T^{(\ell)}_n} \phi = \frac{n! [z^{n-1}] (C'_\zeta(z))^2}{n![z^n]C_\zeta(z)} \sim \frac{\frac{\gamma^2r^{-n+1}}{4r^2}}{\frac{\gamma}{2\sqrt{\pi n^3}}r^{-n}} = \frac{\gamma \sqrt{\pi}}{2r} n^{3/2} = \sqrt{\frac{\pi}{2rG_\zeta''(s)}} n^{3/2}.\qedhere\]
\end{proof}

\bibliographystyle{amsplain}
\bibliography{../Probability_Bibliography}
\end{document}